\documentclass{amsart}
\usepackage{amsmath,amssymb,amsthm,color,comment}
\usepackage{tikz, here}
\usetikzlibrary{cd}
\usepackage{hyperref}
\usepackage{cleveref}
\usepackage{caption}

\newtheorem{dfn}{Definition}[section]
\newtheorem{thm}[dfn]{Theorem}

\newtheorem{lem}[dfn]{Lemma}
\newtheorem{cor}[dfn]{Corollary}
\newtheorem{prop}[dfn]{Proposition}
\newtheorem{rem}[dfn]{Remark}

\newtheorem{conj}[dfn]{Conjecture}
\newtheorem{claim}{Claim}[dfn]

\crefname{dfn}{Definition}{Definitions}
\crefname{thm}{Theorem}{Theorems}
\crefname{fact}{Fact}{Facts}
\crefname{lem}{Lemma}{Lemmas}
\crefname{cor}{Corollary}{Corollaries}
\crefname{prop}{Proposition}{Propositions}
\crefname{rem}{Remark}{Remarks}
\crefname{notation}{Notation}{Notations}
\crefname{claim}{Claim}{Claims}
\crefname{conj}{Conjecture}{Conjectures}
\crefname{figure}{Figure}{Figures}

\renewcommand{\subset}{\subseteq}
\newcommand{\power}{\wp}
\newcommand{\uphar}{\mathbin{\upharpoonright}}
\newcommand{\HOD}{\mathrm{HOD}}

\DeclareMathOperator{\cHull}{cHull}

\DeclareMathOperator{\cf}{cf}
\DeclareMathOperator{\crit}{crit}
\DeclareMathOperator{\lh}{lh}
\DeclareMathOperator{\ord}{Ord}
\DeclareMathOperator{\ult}{Ult}

\title{On $\omega$-strongly measurable cardinals in $\mathbb{P}_\mathrm{max}$ extensions}

\author{Navin Aksornthong}
\address{Navin Aksornthong, Hill Center - Busch Campus, 110 Frelinghuysen Road, Piscataway, NJ 08854-8019, USA.}
\email{navin.aksornthong@rutgers.edu}

\author{Takehiko Gappo}
\address{Takehiko Gappo, Institut für Diskrete Mathematik und Geometrie, TU Wien, Wiedner Hauptstrasse 8-10/104, 1040 Wien, Austria.}
\email{takehiko.gappo@tuwien.ac.at}

\author{James Holland}
\address{James Holland}
\email{jch258@scarletmail.rutgers.edu}

\author{Grigor Sargsyan}
\address{Grigor Sargsyan, IMPAN, Antoniego Abrahama 18, 81-825 Sopot, Poland.}
\email{gsargsyan@impan.pl}

\date{July 25, 2023}

\begin{document}
	\maketitle

    \begin{abstract}
        We show that in the $\mathbb{P}_{\mathrm{max}}$ extension of a certain Chang-type model of determinacy, if $\kappa\in\{\omega_1, \omega_2, \omega_3\}$, then the restriction of the club filter on $\kappa\cap\mathrm{Cof}(\omega)$ to $\HOD$ is an ultrafilter in $\HOD$. This answers Question 4.11 of \cite{omega-strongly_meas} raised by Ben-Neria and Hayut. 
    \end{abstract}
    
	\section{Introduction}

    Recently there has been a great deal of interest in models which are obtained via forcing over models of determinacy and in which $\omega_3$ exhibits interesting combinatorial properties. Woodin's seminal \cite[Chapter 9]{AD_ForcingAxiom_NSIdeal} builds models in which  $\omega_2$ has a rich combinatorial structure, but \cite{AD_ForcingAxiom_NSIdeal} does not investigate the combinatorial structure of $\omega_3$ in $\mathbb{P}_{\mathrm{max}}$ extensions. It seems that \cite{PmaxSquare} is the first paper that initiates the study of $\mathbb{P}_{\mathrm{max}}$ extensions in which $\omega_3$ is combinatorially rich. This work was later continued in \cite{larson2021failures}. The current work studies a subclass of the model introduced in \cite{CoveringChang} in which a weaker version of Woodin's $\HOD$ conjecture fails at $\omega_3$.

    Woodin's $\HOD$ dichotomy theorem roughly says that if there is an extendible cardinal, $\HOD$ is either very close to $V$ or very far from $V$, like Jensen's covering lemma for $L$. One of the formulations is as follows.

    \begin{thm}[Woodin; $\HOD$ dichotomy theorem, \cite{HOD_dichotomy}]\label{HOD_dichotomy}
    Suppose that $\delta$ is extendible.
    Then exactly one of the following holds.
    \begin{enumerate}
        \item Every singular cardinal $\lambda$ above $\delta$ is singular in $\HOD$ and $(\lambda^+)^{\HOD}=\lambda^+$.
        \item Every regular cardinal above $\delta$ is measurable in $\HOD$.
    \end{enumerate}
    \end{thm}

    If the second case holds in \cref{HOD_dichotomy}, regular cardinals above $\delta$ are measurable in $\HOD$ in a strong sense, namely, $\omega$-strongly measurable in $\HOD$.
    
    \begin{dfn}[Woodin, \cite{suitable_extender_I}\footnote{The terminology in (1) is due to Ben-Neria and Hayut (\cite{omega-strongly_meas}).}]
        Let $\kappa$ be an uncountable regular cardinal and let $S\subset\kappa$ be a stationary subset with $S\in\HOD$.
        \begin{enumerate}
            \item Let $\eta$ be a cardinal of $\HOD$. Then $\kappa$ is \emph{$(S, {<}\eta)$-strongly measurable in $\HOD$} if there is no partition in $\HOD$ of $S$ into $\eta$ many disjoint stationary subsets of $\kappa$.
            \item $\kappa$ is called \emph{$\omega$-strongly measurable in $\HOD$} if it is $(\kappa\cap\mathrm{Cof}(\omega), {<}\eta)$-strongly measurable for some $\HOD$-cardinal $\eta$ such that $(2^{\eta})^{\HOD}<\kappa$, where $\mathrm{Cof}(\omega)$ denotes the class of ordinals of countable cofinality.
        \end{enumerate}
    \end{dfn}

    Woodin conjectured that the first case in \cref{HOD_dichotomy} always holds.

    \begin{conj}[Woodin; $\HOD$ conjecture, \cite{suitable_extender_I}]
        There is a proper class of uncountable regular cardinals that are not $\omega$-strongly measurable in $\HOD$.
    \end{conj}

    Recently, there is a remarkable progress regarding $\HOD$ conjecuture due to Ben-Neria and Hayut.

    \begin{thm}[Ben-Neria \& Hayut, \cite{omega-strongly_meas}]\label{Ben-neira--Hayut}
        It is consistent relative to an inaccessible cardinal $\theta$ for which $\{o(\kappa)\mid\kappa<\delta\}$\footnote{Here, $o(\kappa)$ denotes the Mitchell order of $\kappa$.} is unbounded in $\theta$ that all successors of regular cardinals are $\omega$-strongly measurable in $\HOD$.
    \end{thm}
    
    It is worth noting that prior to \cite{omega-strongly_meas}, it was not even known if consistently four cardinals can be $\omega$-strongly measurable in $\HOD$ at the same time.

    Now we consider the existence of uncountable cardinals $\kappa$ that are $(\kappa\cap\mathrm{Cof}(\omega), {<}2)$-strongly measurable in $\HOD$, or equivalent the restriction of the club filter on $\kappa\cap\mathrm{Cof}(\omega)$ to $\HOD$ is an ultrafilter in $\HOD$.
    Note that this property is stronger than being $\omega$-strongly measurable in $\HOD$.
    In the model constructed in \cref{Ben-neira--Hayut}, any $\kappa>\omega_1$ is not $(\kappa\cap\mathrm{Cof}(\omega), {<}2)$-strongly measurable in $\HOD$.
    Ben Neria and Hayut wrote in \cite{omega-strongly_meas} that the referee of the paper pointed out that assuming $L(\mathbb{R})\models\mathsf{AD}$, in the $\mathbb{P}_{\mathrm{max}}$ extension of $L(\mathbb{R})$, if $\kappa$ is either $\omega_1$ or $\omega_2$, then $\kappa$ is $(\kappa\cap\mathrm{Cof}(\omega), {<}2)$-strongly measurable in $\HOD$. In this $\mathbb{P}_{\mathrm{max}}$ extension, however, $\omega_3$ is not measurable in $\HOD$.
    Ben Neria and Hayut then ask (see \cite[Question 4.11]{omega-strongly_meas}) if $\omega_3$ can be $(\omega_3\cap\mathrm{Cof}(\omega), {<}2)$-strongly measurable in $\HOD$. We give an affirmative answer to this question by establishing the following theorem.
    
    \begin{thm}\label{MainTheorem}
        It is consistent relative to a Woodin limit of Woodin cardinals that $\mathsf{ZFC}$ holds and if $\kappa\in\{\omega_1, \omega_2, \omega_3\}$, then $\kappa$ is $(\kappa\cap\mathrm{Cof}(\omega), {<}2)$-strongly measurable in $\HOD$, i.e., the restriction of the club filter on $\kappa\cap\mathrm{Cof}(\omega)$ to $\HOD$ is an ultrafilter in $\HOD$.
    \end{thm}

    As we have already mentioned in the first paragraph, the model of \cref{MainTheorem} is built as the $\mathbb{P}_{\mathrm{max}}$ extension of a subclass introduced in \cite{CoveringChang}.
    Our main task is to show that in the determinacy model, the restriction of the club filter on $\Theta\cap\mathrm{Cof}(\omega)$ to $\HOD$ is an ultrafilter in $\HOD$.
    To show this, we make use of condensing sets, which was originally introduced for core model induction at and beyond the level of ``$\mathsf{AD}_{\mathbb{R}}+\Theta$ is regular.''

    \subsection*{Acknowledgments}
    The first, second, and third authors thank IMPAN for hosting them in 2022 when part of this work was done during their stay.
    The second author thanks Daisuke Ikegami for giving helpful comments and finding errors when he gave a talk on this topic at Waseda Set Theory Seminar in April 2023.
    He is also grateful to Kenta Tsukuura and Toshimichi Usuba for organizing the seminar and affording him the opportunity to present his work.
    
    The second author was supported by Elise Richter grant number V844 and international  grant number I6087 of the Austrian Science Fund (FWF). The fourth author's work is funded by the National Science Center, Poland under the Weave-UNISONO
call in the Weave programme, registration number UMO-2021/03/Y/ST1/00281.
    
    \section{Analysis of a Chang-type model of determinacy}

    \subsection{Definitions and Notations}

    We choose a subclass of a determinacy model introduced in \cite{CoveringChang} as a ground model for the $\mathbb{P}_{\mathrm{max}}$ forcing. The model is constructed in a symmetric extension of a certain hod premouse. Roughly speaking, a hod premouse is a structure of the form $L_{\alpha}[\vec{E}, \Sigma]$, where $\vec{E}$ is a coherent sequence of extenders and $\Sigma$ is a fragment of its own iteration strategy. \footnote{A hod premouse is designed for representing $\mathrm{HOD}$ of a determinacy model of the form $L(\power(\mathbb{R}))$, which is why the name includes ``hod.''} A hod pair is a pair of a hod premouse and its iteration strategy with some regularity properties. In this paper, we use Steel's least branch (lbr) hod premouse, which theory is developed in \cite{CPMP}. See \cite[Definition 9.2.2]{CPMP} for the precise definition of a hod pair.

    To avoid including $\mathsf{AD}^+$ in our background theory, we need to assume more regularity of the iteration strategy in a hod pair that follows from $\mathsf{AD}^+$. According to \cite{CoveringChang}, we say that a hod pair $(\mathcal{V}, \Omega)$ is \emph{excellent} if $\mathcal{V}$ is countable, $\Omega$ is $(\omega_1, \omega_1+1)$-iteration strategy for $\mathcal{V}$, and whenever $\mathcal{P}\trianglelefteq\mathcal{V}, o(\mathcal{P})$ is an inaccessible cardinal of $\mathcal{V}$, $\rho(\mathcal{V})>o(\mathcal{P})$, and $\Sigma=\Omega_{\mathcal{P}}$, then the following hold:
    \begin{enumerate}
        \item $\Sigma$ admits full normalization, i.e., whenever $\mathcal{T}$ is an iteration tree on $\mathcal{P}$ via $\Sigma$ with last model $\mathcal{Q}$, there is a normal iteration $\mathcal{U}$ on $\mathcal{P}$ via $\Sigma$ with last model $\mathcal{Q}$ such that $\pi^{\mathcal{T}}$ exists if and only if $\pi^{\mathcal{U}}$ exists, and if $\pi^{\mathcal{T}}$ exists then $\pi^{\mathcal{T}}=\pi^{\mathcal{U}}$,
        \item $\Sigma$ is positional, i.e., if $\mathcal{Q}$ is a $\Sigma$-iterate of $\mathcal{P}$ via an iteration tree $\mathcal{T}$ and it is also via another iteration tree $\mathcal{U}$, then $\Sigma_{\mathcal{T}, \mathcal{Q}}=\Sigma_{\mathcal{U}, \mathcal{Q}}$,\footnote{We then are allowed to denote the unique tail strategy for $\mathcal{Q}$ by $\Sigma_{\mathcal{Q}}$.}
        \item $\Sigma$ is directed, i.e., if $\mathcal{Q}_0$ and $\mathcal{Q}_1$ are $\Sigma$-iterates of $\mathcal{P}$ via iteration trees above some ordinal $\eta$, then there is an $\mathcal{R}$ such that $\mathcal{R}$ is a $\Sigma_{\mathcal{Q}_i}$-iterate of $\mathcal{Q}_i$ via an iteration tree above $\eta$ for any $i\in\{0, 1\}$,
        \item $(\mathcal{P}, \Sigma)$ satisfies generic interpretability in the sense of \cite[Theorem 11.1.1]{CPMP}, and
        \item $\Sigma$ is segmentally normal, i.e., whenever $\eta$ is inaccessible cardinal of $\mathcal{P}$ such that $\rho(\mathcal{P})>\eta$, $\mathcal{Q}$ is a non-dropping $\Sigma$-iterate of $\mathcal{P}$ via an iteration tree $\mathcal{T}$ that is above $\eta$, and $\mathcal{R}$ is a non-dropping $\Sigma_{\mathcal{Q}}$-iterate of $\mathcal{Q}$ via an iteration tree $\mathcal{U}$ that is based on $\mathcal{Q}\vert\eta$, then $\Sigma_{\mathcal{P}\vert\eta}=(\Sigma_{\mathcal{Q}})_{\mathcal{P}\vert\eta}$ and letting $\mathcal{R}^*$ be a non-dropping $\Sigma$-iterate of $\mathcal{P}$ via the iteration tree $\mathcal{U}^*$ that has the same extenders and branches as $\mathcal{U}$, $\mathcal{R}$ is a non-dropping $\Sigma_{\mathcal{R}^*}$-iterate of $\mathcal{R}^*$ via a normal iteration tree that is above $\pi_{\mathcal{P}, \mathcal{R}^*}(\eta)$.
    \end{enumerate}
    Siskind and Steel showed that under $\mathsf{AD}^+$, every countable hod pair is excellent (\cite{CPMP, full_normalization}).
    Our definition of excellence has slight differences from \cite{CoveringChang}.
    First, we omit stability and pullback consistency from the definition because they are already part of the definition of a hod pair in \cite{CPMP}.
    Also, we do not mention to strongly non-dropping iteration trees, simply because it turns out that we do not have to.
    See the remark after \cref{genericity_iteration} as well.
    The consequence of excellence that the reader should be particularly aware of is that if a hod pair $(\mathcal{V}, \Omega)$ is excellent, then 
    \begin{itemize}
        \item for any $\mathcal{P}$ and $\Sigma$ as in the definition of excellence, $\Sigma$ has a canonical extension $\Sigma^g$ in $\mathcal{P}[g]$, where $g\subset\mathrm{Col}(\omega, {<}\delta)$ is $\mathcal{P}$-generic and $\delta$ is the supremum of all Woodin cardinals of $\mathcal{P}$, and
        \item internal direct limit models as defined in \cref{dfn:internal_direct_limit_system} are well-defined.
    \end{itemize}

    Now we describe our setup, which is the same as in \cite{CoveringChang}. Let $(\mathcal{V}, \Omega)$ be an excellent hod pair such that $\mathcal{V}\models\mathsf{ZFC}$. Suppose that in $\mathcal{V}$, $\delta$ is a cardinal that is a limit of Woodin cardinals and if $\delta$ is not regular, then its cofinality is not measurable. \footnote{Throughout this paper, we adopt the following standard convention: if $\mathcal{M}$ is an lbr hod premouse, then ``$\delta$ has some large cardinal property in $\mathcal{M}$'' actually means ``the extender sequence of $\mathcal{M}$ witnesses that $\delta$ has some large cardinal property in $\mathcal{M}$.''} We let $\mathcal{P}=\mathcal{V}\vert(\delta^+)^{\mathcal{V}}$ and let $\Sigma$ be the $(\omega, \delta+1)$-iteration strategy for $\mathcal{P}$ determined by the strategy predicate of $\mathcal{V}$. Also, let $g\subset\mathrm{Col}(\omega, {<}\delta)$ be $\mathcal{V}$-generic. We fix the objects defined in this paragraph throughout the paper and work in $\mathcal{V}[g]$ unless otherwise noted.

    We denote the derived model at $\delta$ (computed in $\mathcal{V}[g]$) by $\mathsf{DM}$. More precisely, let $\mathbb{R}^*_g = \bigcup_{\alpha<\delta} \mathbb{R}^{\mathcal{V}[g\uphar\alpha]}$, where $g\uphar\alpha:=g\cap\mathrm{Col}(\omega, {<}\alpha)$, and let
    \[
    \Gamma^*_g = \{A^*_g\subset\mathbb{R}^*_g\mid\exists\alpha<\delta( A\subset\mathbb{R}^{\mathcal{V}[g\uphar\alpha]}\land \mathcal{V}[g\uphar\alpha]\models\text{$A$ is ${<}\delta$-universally Baire})\}.
    \]
    Here, we write $A^*_g = \bigcup_{\beta\in(\alpha, \delta)} A^{g\uphar\beta}$, where $A^{g\uphar\beta}$ is the canonical extension of $A$ in $\mathcal{V}[g\uphar\beta]$ via its ${<}\delta$-universally Baire representation. We define $\mathsf{DM}=L(\Gamma^*_g, \mathbb{R}^*_g)$. See \cite{St09} for basic properties of $\mathsf{DM}$.

    We define $I^*_g(\mathcal{P}, \Sigma)$ as the set of all non-dropping \footnote{We say that $\mathcal{Q}$ is a non-dropping iterate of $\mathcal{P}$ via $\mathcal{T}$ if the main branch of $\mathcal{T}$ does not drop.} $\Sigma$-iterates of $\mathcal{P}$ via an $(\omega, \delta+1)$-iteration tree $\mathcal{T}$ of $\mathcal{P}$ based on $\mathcal{P}\vert\delta$ \footnote{For an iteration tree $\mathcal{T}$ on $\mathcal{P}$, we say that $\mathcal{T}$ is based on $\mathcal{P}\vert\delta$ if it only uses extenders on the extender sequence of $\mathcal{P}\vert\delta$ and their images.} such that $\pi^{\mathcal{T}}(\delta)=\delta$ and $\mathcal{T}\in\mathcal{V}[g\uphar\xi]$ for some $\xi<\delta$. Let $\mathcal{Q}\in I^*_g(\mathcal{P}, \Sigma)$. Because $\Sigma$ (and its canonical extensions to generic extensions) admits full normalization, $\mathcal{Q}$ is a non-dropping normal $\Sigma$-iterate of $\mathcal{P}$. So, let $\mathcal{T}_{\mathcal{P}, \mathcal{Q}}$ be a unique normal iteration tree of $\mathcal{P}$ via $\Sigma$ with last model $\mathcal{Q}$. Note that the length of $\mathcal{T}_{\mathcal{P}, \mathcal{Q}}$ is at most $\delta+1$. Let $\Sigma_{\mathcal{Q}}$ be the tail strategy $\Sigma_{\mathcal{Q}, \mathcal{T}_{\mathcal{P}, \mathcal{Q}}}$. Since $\Sigma$ is positional, $\Sigma_{\mathcal{Q}}=\Sigma_{\mathcal{Q}, \mathcal{U}}$ for any $\Sigma$-iteration tree $\mathcal{U}$ on $\mathcal{P}$ with last model $\mathcal{Q}$. Let $\pi_{\mathcal{P}, \mathcal{Q}}\colon\mathcal{P}\to\mathcal{Q}$ be the iteration map via $\mathcal{T}_{\mathcal{P}, \mathcal{Q}}$. Moreover, since $\mathcal{V}$ does not project across $(\delta^+)^{\mathcal{V}}$, we can apply $\mathcal{T}_{\mathcal{P}, \mathcal{Q}}$ to $\mathcal{V}$ according to $\Omega$. Then let $\mathcal{V}_{\mathcal{Q}}$ be the last model of $\mathcal{T}_{\mathcal{P}, \mathcal{Q}}$ when it is applied to $\mathcal{V}$. It is not hard to see that $\mathcal{Q}=\mathcal{V}_{\mathcal{Q}}\vert(\delta^+)^{\mathcal{V}_{\mathcal{Q}}}$ and $\Sigma_{\mathcal{Q}}$ is compatible with the strategy predicate of $\mathcal{V}_{\mathcal{Q}}$.

    \begin{dfn}\label{dfn:internal_direct_limit_system}
        For any $\mathcal{Q}\in I^*_{g}(\mathcal{P}, \Sigma)$, we define $\mathcal{F}^*_g(\mathcal{Q})$ as the set of all non-dropping $\Sigma_{\mathcal{Q}}$-iterates $\mathcal{R}$ of $\mathcal{Q}$ such that $\lh(\mathcal{T}_{\mathcal{Q}, \mathcal{R}})<\delta$, $\mathcal{T}_{\mathcal{Q}, \mathcal{R}}$ is based on $\mathcal{P}\vert\delta$, and $\mathcal{T}_{\mathcal{Q}, \mathcal{R}}\in\mathcal{V}[g\uphar\xi]$ for some $\xi<\delta$. Since $\Sigma$ is directed, $\mathcal{F}^*_g(\mathcal{Q})$ can be regarded as a direct limit system under iteration maps. We define $\mathcal{M}_{\infty}(\mathcal{Q})$ as the direct limit model of the system $\mathcal{F}^*_g(\mathcal{Q})$. For any $\mathcal{R}\in\mathcal{F}^*_g(\mathcal{Q})$, let $\pi^{\mathcal{Q}}_{\mathcal{R}, \infty}\colon\mathcal{R}\to\mathcal{M}_{\infty}(\mathcal{Q})$ be the direct limit map and let $\pi_{\mathcal{Q}, \infty}=\pi^{\mathcal{Q}}_{\mathcal{Q}, \infty}$.
        Finally, let $\delta^{\mathcal{Q}}_{\infty}=\pi_{\mathcal{Q}, \infty}(\delta)$.
    \end{dfn}

    Let $\mathcal{Q}\in I^*_g(\mathcal{P}, \Sigma)$. Since any iteration tree based on $\mathcal{Q}\vert\delta$ can be applied to $\mathcal{V}_{\mathcal{Q}}$, we can similarly define a direct limit system $\mathcal{F}^*_g(\mathcal{V}_{\mathcal{Q}})$, which consists of models $\mathcal{V}_{\mathcal{R}}$ and iteration maps $\pi_{\mathcal{V}_{\mathcal{R}}, \mathcal{V}_{\mathcal{R}^*}}$, where $\mathcal{R}, \mathcal{R}^*\in\mathcal{F}^*_g(\mathcal{Q})$ are such that $\mathcal{R}^*$ is a non-dropping iterate of $\mathcal{R}$. It is not hard to see that $\mathcal{V}_{\mathcal{M}_{\infty}(\mathcal{Q})}$ is the direct limit model of $\mathcal{F}^*_g(\mathcal{V}_{\mathcal{Q}})$. For any $\mathcal{R}\in\mathcal{F}^*_g(\mathcal{V}_{\mathcal{R}})$, let $\pi^{\mathcal{Q}}_{\mathcal{V}_{\mathcal{R}}, \infty}\colon\mathcal{V}_{\mathcal{R}}\to\mathcal{V}_{\mathcal{M}_{\infty}(\mathcal{Q})}$ be the corresponding direct limit map which extends $\pi^{\mathcal{Q}}_{\mathcal{R}, \infty}\colon\mathcal{R}\to\mathcal{M}_{\infty}(\mathcal{Q})$.

    \begin{dfn}
        The Chang model over the derived model (at $\delta$ computed in $\mathcal{V}[g]$) is defined by
        \[
        \mathsf{CDM}=L(\mathcal{M}_{\infty}, \cup_{\xi<\delta_{\infty}}{}^{\omega}\xi, \Gamma^*_g, \mathbb{R}^*_g),\footnotemark
        \]
        where $\mathcal{M}_{\infty}=\mathcal{M}_{\infty}(\mathcal{P})$ and $\delta_{\infty}=\delta^{\mathcal{P}}_{\infty}$. \footnotetext{In \cite{CoveringChang}, this model is denoted by $C(g)$.}
        We also define
        \[
        \mathsf{CDM}^{-}=L(\mathcal{M}_{\infty}\vert(\Theta^+)^{\mathcal{M}_{\infty}}, \power_{\omega_1}(\mathcal{M}_{\infty}\vert(\Theta^+)^{\mathcal{M}_{\infty}}), \Gamma^*_g, \mathbb{R}^*_g),
        \]
        where $\Theta=\Theta^{\mathsf{CDM}}$.
        \footnote{In general, $\Theta^{\mathsf{CDM}}\leq\delta_{\infty}$ holds, but we will see that $\Theta^{\mathsf{CDM}}<\delta_{\infty}$ under the assumption of \cref{main_thm_on_CDM}. Also, by \cref{sets_of_reals_in_CDM}, $\Theta^{\mathsf{CDM}}=\Theta^{\mathsf{CDM}^-}$.}
    \end{dfn}

    Clearly, $\mathsf{DM}\subset\mathsf{CDM}^{-}\subset\mathsf{CDM}\subset\mathcal{V}[g]$. We choose $\mathsf{CDM}^{-}$ rather than $\mathsf{CDM}$ as a ground model for the $\mathbb{P}_{\mathrm{max}}$ extension.
    We will show the following theorem in the rest of this section.
    
    \begin{thm}\label{main_thm_on_CDM}
        Suppose that $(\mathcal{V}, \Omega)$ is an excellent hod pair and $\kappa<\delta$ are cardinals of $\mathcal{V}$ such that
        \begin{align*}
        \mathcal{V}\models\mathsf{ZFC} & +\delta\text{ is a regular limit of Woodin cardinals}\\
        & + \kappa\text{ is the least ${<}\delta$-strong cardinal}.
        \end{align*}
        Let $g\subset\mathrm{Col}(\omega, {<}\delta)$ be $\mathcal{V}$-generic.
        Then $\mathsf{CDM}^{-}$ (computed in $\mathcal{V}[g]$) satisfies the following:
        \begin{enumerate}
            \item $\mathsf{AD}^{+}+\mathsf{AD}_{\mathbb{R}}+\mathsf{DC}+\Theta$ is regular.
            \item Let $\kappa_{\infty}=\pi_{\mathcal{P}, \infty}(\kappa)$. Then $\Theta=\kappa_{\infty}$ and $\HOD\|(\Theta^+)^{\HOD}=\mathcal{M}_{\infty}\vert(\kappa_{\infty}^+)^{\mathcal{M}_{\infty}}$.\footnote{For any transitive model $M$ of $\mathsf{ZF}$ without fine structure and any $\alpha\in \ord\cap M$, we write $M\|\alpha$ for $V_{\alpha}^{M}$.}
            \item The restriction of the club filter on $\Theta\cap\mathrm{Cof}(\omega)$ to $\HOD$ is a normal $\Theta$-complete ultrafilter in $\HOD$.
        \end{enumerate}
    \end{thm}

    Our proof shows that \cref{main_thm_on_CDM} is also true for $\mathsf{CDM}$, but we do not need this fact for the proof of \cref{MainTheorem}.
    The fourth author recently showed that the consistency strength of the assumption of \cref{main_thm_on_CDM} is strictly weaker than a Woodin limit of Woodin cardinals.
    
    \subsection{Lemmas from \cite{CoveringChang}}

    We summarize the results of \cite{CoveringChang} in this subsection.
    We assume that $(\mathcal{V}, \Omega)$ is an excellent hod pair with $\mathcal{V}\models\mathsf{ZFC}$ and that in $\mathcal{V}$, $\delta$ is a cardinal that is limit of Woodin cardinals and if $\delta$ is not regular, then its cofinality is not measurable.
    
    Let $\mathcal{M}$ be an lbr hod premouse. Then we say that an open interval of ordinals $(\eta, \delta)$ is a \emph{window of $\mathcal{M}$} if in $\mathcal{M}$, $\eta$ is an inaccessible cardinal and $\delta$ is the least Woodin cardinal above $\eta$ in $\mathcal{M}$. For any iteration tree $\mathcal{T}$ on $\mathcal{M}$, we say that \emph{$\mathcal{T}$ is based on a window $(\eta, \delta)$} if it is based on $\mathcal{M}\vert\delta$ and is above $\eta$, i.e., $\mathcal{T}$ uses only extenders on the extender sequence of $\mathcal{M}\vert\delta$ with critical point $>\eta$ and their images. Also, a sequence $\langle w_{\alpha}\mid\alpha<\lambda\rangle$ of windows of $\mathcal{M}$ is \emph{increasing} if whenever $\alpha<\beta$, $\sup(w_{\alpha})\leq\inf(w_{\beta})$.

    \begin{dfn}\label{window-based_iteration}
        Let $\mathcal{Q}\in I^*_g(\mathcal{P}, \Sigma)$ and let $\mathcal{R}\in I^*_g(\mathcal{Q}, \Sigma_{\mathcal{Q}})$. We say that $\mathcal{R}$ is a \emph{window-based iterate of $\mathcal{Q}$} if there is an $\xi<\delta$ such that $\mathcal{R}\in\mathcal{V}[g\uphar\xi]$, an increasing sequence of windows $\langle w_{\alpha}\mid\alpha<\cf(\delta)\rangle$ of $\mathcal{R}$ and a sequence $\langle\mathcal{R}_{\alpha}\mid\alpha\leq\cf(\delta)\rangle$ of lbr hod premice in $\mathcal{V}[g\uphar\xi]$ such that
        \begin{enumerate}
            \item $\delta=\sup\{\sup(w_{\alpha})\mid\alpha<\cf(\delta)\}$.
            \item $\mathcal{R}_0$ is a non-dropping iterate of $\mathcal{Q}$ based on $\mathcal{Q}\vert\inf(w_0)$.
            \item $\mathcal{R}_{\alpha+1}$ is a non-dropping iterate of $\mathcal{R}_{\alpha}$ based on a window $\pi_{\mathcal{Q}, \mathcal{R}_{\alpha}}(w_{\alpha})$.
            \item for any limit ordinal $\lambda\leq\cf(\delta)$, $\mathcal{R}_{\lambda}$ is the direct limit of $\langle\mathcal{R}_{\alpha}, \pi_{\mathcal{R}_{\alpha}, \mathcal{R}_{\beta}}\mid\alpha<\beta<\lambda\rangle$.
            \item $\mathcal{R}=\mathcal{R}_{\cf(\delta)}$.
        \end{enumerate}
    \end{dfn}

    Let $\mathcal{M}$ be an lbr hod premouse. An extender $E\in\vec{E}^{\mathcal{M}}$ is called \emph{nice} if the supremum of the generators of $E$ is an inaccessible cardinal in $\mathcal{M}$. For any window $(\eta, \delta)$ of $\mathcal{R}$, let $\mathsf{EA}^{\mathcal{M}}_(\eta, \delta)$ be Woodin's extender algebra with $\omega$ generators at $\delta$ in $\mathcal{M}$ that only uses nice extenders $E\in\vec{E}^{\mathcal{M}\vert\eta}$ such that $\crit(E)>\eta$, see \cite{Farah} and \cite{OIMT}.

    \begin{dfn}\label{genericity_iteration}
        Let $\mathcal{Q}\in I^*_g(\mathcal{P}, \Sigma)$ and let $\mathcal{R}\in I^*_g(\mathcal{Q}, \Sigma_{\mathcal{Q}})$. We say that $\mathcal{R}$ is a \emph{genericity iterate of $\mathcal{Q}$} if it is a window-based iterate of $\mathcal{Q}$ as witnessed by $\langle w_{\alpha}\mid\alpha<\cf(\delta)\rangle$ and $\langle\mathcal{R}_{\alpha}\mid\alpha\leq\cf(\delta)\rangle$ such that
        \begin{enumerate}
            \item for any $x\in\mathbb{R}^{\mathcal{P}[g]}$, there is an $\alpha<\delta$ such that $x$ is $\mathsf{EA}^{\mathcal{R}}_{\pi_{\mathcal{Q}, \mathcal{R}}(w_{\alpha})}$-generic over $\mathcal{R}$.
            \item for any $\alpha<\cf(\delta)$, $w_{\alpha}\in\mathrm{ran}(\pi_{\mathcal{Q}, \mathcal{R}})$.
        \end{enumerate}
        We say that $\mathcal{R}$ is a \emph{genericity iterate of $\mathcal{Q}$ above $\eta$} if it is a genericity iterate of $\mathcal{Q}$ witnessed by $\langle w_{\alpha}\mid\alpha<\cf(\delta)\rangle$ and $\langle\mathcal{R}_{\alpha}\mid\alpha\leq\cf(\delta)\rangle$ such that $\inf(w_0)\geq\eta$.
    \end{dfn}

    In \cite{CoveringChang}, a genericity iteration is required to be strongly non-dropping, or use only nice extenders. This condition is actually redundant, so we omit it from \cref{genericity_iteration}.

    \begin{lem}[{\cite[Propositions 3.3 and 3.4]{CoveringChang}}]\label{easy_lemma}\leavevmode
    \begin{enumerate}
        \item For any $\mathcal{P}^*\in\mathcal{F}^*_{g}(\mathcal{P})$ and any $\eta<\delta$, there is a genericity iterate $\mathcal{Q}$ of $\mathcal{P}$ and $\crit(\pi_{\mathcal{P}^*, \mathcal{Q}})>\eta$, and $\mathcal{T}_{\mathcal{P}, \mathcal{P}^*}{}^{\frown}\mathcal{T}_{\mathcal{P}^*, \mathcal{Q}}$ is a normal iteration tree.
        \item If $\mathcal{Q}$ is a genericity iterate of $\mathcal{P}$ and $\mathcal{R}$ is a genericty iterate of $\mathcal{Q}$, then $\mathcal{R}$ is a genericity iterate of $\mathcal{P}$.
    \end{enumerate}
    \end{lem}

    \begin{lem}[{\cite[Theorem 3.8]{CoveringChang}}]\label{pres_of_M_infty}
        For any genericity iterate $\mathcal{Q}$ of $\mathcal{P}$,
        \[
        \mathcal{M}_{\infty}(\mathcal{P})=\mathcal{M}_{\infty}(\mathcal{Q}).
        \]
        Moreover, $\pi_{\mathcal{P}, \infty}=\pi_{\mathcal{Q}, \infty}\circ\pi_{\mathcal{P}, \mathcal{Q}}$. In particular, $\delta^{\mathcal{P}}_{\infty}=\delta_{\infty}^{\mathcal{Q}}$.
    \end{lem}

    Now let $\mathcal{Q}$ be a genericity iterate of $\mathcal{P}$. Then there is a $\mathcal{Q}$-generic $h\subset\mathrm{Col}(\omega, {<}\delta)$ (in $\mathcal{V}[g]$) such that $(\mathbb{R}^*_g)^{\mathcal{P}[g]}=(\mathbb{R}^*_h)^{\mathcal{Q}[h]}$. We call such an $h$ \emph{maximal}.

    \begin{lem}[{\cite[Proposition 4.2]{CoveringChang}}]\label{pres_of_CDM}
        Let $\mathcal{Q}$ be a genericity iterate of $\mathcal{P}$ and let $h\subset\mathrm{Col}(\omega, <\delta)$ be a maximal $\mathcal{Q}$-generic such that $h\in\mathcal{V}[g]$. Then
        \[
        \mathsf{CDM}=\mathsf{CDM}^{\mathcal{V}_{\mathcal{Q}}[h]}\text{ and } \mathsf{CDM}^{-}=(\mathsf{CDM}^{-})^{\mathcal{V}_{\mathcal{Q}}[h]}.
        \]
    \end{lem}

    The following theorem is the main result of \cite{CoveringChang}.

    \begin{thm}[\cite{CoveringChang}]\label{sets_of_reals_in_CDM}
        $\mathsf{CDM}\cap\power(\mathbb{R}^*_g)=\mathsf{CDM}^{-}\cap\power(\mathbb{R}^*_g)=\Gamma^*_g$.
    \end{thm}

    \begin{cor}\label{determinacy_in_CDM_1}
        Both $\mathsf{CDM}$ and $\mathsf{CDM}^-$ are models of $\mathsf{AD}^{+}+\mathsf{AD}_{\mathbb{R}}$.
    \end{cor}
    \begin{proof}
        As \cite[Theorem 11.3.2]{CPMP}, Steel showed that the sets of reals in $\mathsf{DM}$ is $\Gamma^*_g$ and thus $\mathsf{DM}\models\mathsf{AD}^{+}+\mathsf{AD}_{\mathbb{R}}$.\footnote{In \cite{dm_of_self_it}, the second and fourth authors also showed the same conclusion for any self-iterable structures, which may not be fine structural.}
        So the corollary follows from \cref{sets_of_reals_in_CDM}.
    \end{proof}

    \subsection{Main proofs}
    
    In this subsection, we prove the properties of $\mathsf{CDM}^-$ listed in \cref{main_thm_on_CDM}. Regarding \cref{determinacy_in_CDM_2} and \cref{value_of_Theta}, the proofs are the same as ones in \cite{CDM+}. From now on, we assume that $(\mathcal{V}, \Omega)$ is an excellent hod pair and $\kappa<\delta$ are cardinals of $\mathcal{V}$ such that
        \begin{align*}
        \mathcal{V}\models\mathsf{ZFC} & +\delta\text{ is a regular limit of Woodin cardinals}\\
        & + \kappa\text{ is the least ${<}\delta$-strong cardinal}.
        \end{align*}
    Also, let $g\subset\mathrm{Col}(\omega, {<}\delta)$ be $\mathcal{V}$-generic and $\mathsf{CDM}^{-}$ is defined in $\mathcal{V}[g]$.
    
    \begin{thm}\label{determinacy_in_CDM_2}
        $\mathsf{CDM}^{-}\models\mathsf{DC}+\Theta\text{ is regular.}$
    \end{thm}
    \begin{proof}
        First, we show that $\mathsf{CDM}^{-}\models\cf(\Theta)>\omega$. It follows from the proof of \cite[Corollary 3.7]{dm_of_self_it} without any change as follows. Suppose toward a contradiction that $\mathsf{CDM}^{-}\models\cf(\Theta)=\omega$. Then by \cref{sets_of_reals_in_CDM}, there is a sequence $\langle A_n \mid n<\omega\rangle$ that is Wadge cofinal in $\Gamma^*_g$. For any $n<\omega$, let $\lambda_n<\delta$ be such that there is an $B_n\subset\mathbb{R}^{\mathcal{V}[g\uphar\lambda_n]}$ such that it is ${<}\delta$-universally Baire in $\mathcal{V}[g\uphar\lambda_n]$ and $A_n=B^*_n$. Let $\lambda=\sup_{n<\omega}\lambda_n$. Since $\delta$ is regular, $\lambda<\delta$. Let $\delta'<\delta$ be the least Woodin cardinal above $\lambda$ in $\mathcal{V}$. Then by \cite[Fact 3.3]{dm_of_self_it}, all $A_n$'s are projective in $\mathrm{Code}(\Sigma^g_{\mathcal{P}\vert\delta'})$. It follows, however, that even if $\delta'<\xi<\delta$, $\mathrm{Code}(\Sigma^g_{\mathcal{P}\vert\xi})$ is projective in $\mathrm{Code}(\Sigma^g_{\mathcal{P}\vert\delta'})$, which contradicts \cite[Lemma 3.4]{dm_of_self_it}.

        Now we can easily show that $\mathsf{DC}$ holds in $\mathsf{CDM}^{-}$. In \cite{Solovay}, Solovay showed that $\mathsf{AD}+\mathsf{DC}_{\mathbb{R}}+\cf(\Theta)>\omega$ implies that $\mathsf{DC}_{\power(\mathbb{R})}$. So, $\mathsf{CDM}^{-}$ satisfies $\mathsf{DC}_{\power(\mathbb{R})}$. Then in $\mathsf{CDM}^{-}$, $\mathsf{DC}$ reduces to $\mathsf{DC}_X$ where $X=\power_{\omega_1}(\mathcal{M}_{\infty}\vert(\Theta^+)^{\mathcal{M}_{\infty}})$, because any element of $\mathsf{CDM}^{-}$ is ordinal definable in parameters from $X$ and sets of reals. Since any $\omega$-sequence from $X$ can be easily coded into an element of $X$, $\mathsf{DC}_X$ in $\mathcal{V}[g]$ implies $\mathsf{DC}_X$ in $\mathsf{CDM}^{-}$. Therefore, $\mathsf{CDM}^{-}\models\mathsf{DC}$.

        The regularity of $\Theta$ in $\mathsf{CDM}^{-}$ also follows from the proof of \cite[Theorem 1.3]{dm_of_self_it}, but we need to use \cref{pres_of_CDM}. Let $\Theta=\Theta^{\mathsf{CDM}^{-}}$. Suppose toward a contradiction that there is a cofinal map $f\colon\mathbb{R}^*_g\to\Theta$ in $\mathsf{CDM}^{-}$. Then there are a formula in the language of set theory, an ordinal $\gamma$, $Y\in\power_{\omega_1}(\mathcal{M}_{\infty}\vert(\Theta^+)^{\mathcal{M}_{\infty}})$, $Z\in\Gamma^*_g$, $x\in\mathbb{R}^*_g$ and $\vec{\beta}\in{}^{<\omega}\gamma$ such that
        \[
        f=\{\langle u, \zeta\rangle\in\mathbb{R}^*_g\times\Theta\mid\mathsf{CDM}^-\vert\gamma\models\phi[u, \zeta, Y, Z, x, \vec{\beta}]\},
        \]
        where $\mathsf{CDM}^{-}\vert\gamma = L_{\gamma}(\mathcal{M}_{\infty}\vert(\Theta^+)^{\mathcal{M}_{\infty}}, \power_{\omega_1}(\mathcal{M}_{\infty}\vert(\Theta^+)^{\mathcal{M}_{\infty}}), \Gamma^*_g, \mathbb{R}^*_g)$.
        We take a genericity iterate $\mathcal{Q}$ of $\mathcal{P}$ such that $\{\vec{\beta}, \gamma\}\cup Y\subset\mathrm{ran}(\pi_{\mathcal{V}_{\mathcal{Q}}, \infty})$ as follows: Let $\mathcal{P}^*\in\mathcal{F}^*_g(\mathcal{P})$ such that $\{\vec{\beta}, \gamma\}\cup Y\subset\mathrm{ran}(\pi^{\mathcal{P}}_{\mathcal{P}^*}, \infty)$. Such a $\mathcal{P}^*$ exists because $\mathcal{F}^*_g(\mathcal{P})$ is countably directed. By \cref{easy_lemma}(1), there is an iterate $\mathcal{Q}$ of $\mathcal{P}^*$ such that it is a genericity iterate of $\mathcal{P}$ and $\mathcal{T}_{\mathcal{P}, \mathcal{P}^*}{}^{\frown}\mathcal{T}_{\mathcal{P}^*, \mathcal{Q}}$ is normal. Since $\pi^{\mathcal{P}}_{\mathcal{V}_{\mathcal{P}^*}, \infty}=\pi^{\mathcal{Q}}_{\mathcal{V}_{\mathcal{Q}}, \infty}\circ\pi_{\mathcal{V}_{\mathcal{P}^*}, \mathcal{V}_{\mathcal{Q}}}$, $\mathcal{Q}$ satisfies the desired property.
    
        \begin{claim}[{\cite[Lemma 4.3]{CoveringChang}}]\label{stabilizing_parameters}
            Whenever $\mathcal{R}$ is a genericity iterate of $\mathcal{Q}$, if $s\in\mathrm{ran}(\pi_{\mathcal{V}_{\mathcal{Q}}, \infty})$ then $\pi_{\mathcal{V}_{\mathcal{Q}}, \mathcal{V}_{\mathcal{R}}}(s)=s$.
        \end{claim}
        
        \begin{proof}
            Let $s_{\mathcal{Q}}=\pi_{\mathcal{V}_{\mathcal{Q}}, \infty}^{-1}(s)$. Then we have
            \[
            \pi_{\mathcal{V}_{\mathcal{Q}}, \mathcal{V}_{\mathcal{R}}}(s)=\pi_{\mathcal{V}_{\mathcal{Q}}, \mathcal{V}_{\mathcal{R}}}(\pi_{\mathcal{V}_{\mathcal{Q}}, \infty}(s_{\mathcal{Q}}))=\pi_{\mathcal{V}_{\mathcal{R}}, \infty}(\pi_{\mathcal{V}_{\mathcal{Q}}, \mathcal{V}_{\mathcal{R}}}(s_{\mathcal{Q}}))=\pi_{\mathcal{V}_{\mathcal{Q}}, \infty}(s_{\mathcal{Q}})=s.
            \]
            The second equation follows from the elementarity of $\pi_{\mathcal{V}_{\mathcal{Q}}, \mathcal{V}_{\mathcal{R}}}$ and the third equation holds since $\pi_{\mathcal{V}_{\mathcal{Q}}, \infty}=\pi_{\mathcal{V}_{\mathcal{R}}, \infty}\circ\pi_{\mathcal{V}_{\mathcal{Q}}, \mathcal{V}_{\mathcal{R}}}$ by \cref{pres_of_M_infty}.
        \end{proof}
        
        Let $h\subset\mathrm{Col}(\omega, {<}\delta)$ be a maximal $\mathcal{Q}$-generic. Let $\xi_Y<\delta$ be such that $Y\subset\pi_{\mathcal{Q}, \infty}[\xi_Y]$. Let $y\in\mathbb{R}^*_{h}$ code a function $f_y\colon\omega\to\xi_Y$ such that $Y=\pi_{\mathcal{Q}, \infty}[\mathrm{ran}(f_y)]$.
        Also, since $\{\mathrm{Code}(\Sigma^g_{\mathcal{P}\vert\xi})\mid\xi<\delta\}$\footnote{For an iteration strategy $\Sigma$ for a countable structure, $\mathrm{Code}(\Sigma)$ is a set of reals that canonically codes $\Sigma\uphar\mathrm{HC}$, where $\mathrm{HC}$ denotes the set of hereditarily countable sets. See \cite[Section 2.7]{CPMP}.} is Wadge cofinal in $\Gamma^*_g$ as argued in the proof of \cite[Proposition 4.2]{CoveringChang}, we may assume that $Z=\mathrm{Code}(\Sigma^g_{\mathcal{P}\vert\xi_Z})$ for some $\xi_Z<\delta$.
        Let $z\in\mathbb{R}^*_g$ be a real coding $\pi_{\mathcal{P}, \mathcal{Q}}\uphar(\mathcal{P}\vert\xi_Z)\colon\mathcal{P}\vert\xi_Z\to\mathcal{Q}\vert\pi_{\mathcal{P}, \mathcal{Q}}(\xi_Z)$. Note that $Z$ can be defined from $z$ as the code of the $\pi_{\mathcal{P}, \mathcal{Q}}$-pullback of the strategy for $\mathcal{Q}\vert\pi_{\mathcal{P}, \mathcal{Q}}(\xi_Z)$ determined by the strategy predicate of $\mathcal{Q}$. Because $\mathsf{CDM}^{-}=(\mathsf{CDM}^{-})^{\mathcal{V}_{\mathcal{Q}}[h]}$ by \cref{pres_of_CDM}, we have
        \[
        f=\{\langle u, \zeta\rangle\in\mathbb{R}^*_g\times\Theta\mid\mathcal{V}_{\mathcal{Q}}[x, y, z][u]\models\phi^*(u, \zeta, x, y, z, \delta, \vec{\beta}, \gamma)\},
        \]
        where a formula $\phi^*$ is the conjunction of the following:
        \begin{itemize}
            \item $y$ codes a function $f\colon\omega\to\xi$ for some $\xi<\delta$, and
            \item $z$ codes an elementary embedding $\pi\colon\mathcal{M}\to\mathcal{N}$ for some lbr hod premice $\mathcal{M}$ and $\mathcal{N}$ with $\mathcal{N}\trianglelefteq\mathcal{Q}$, and 
            \item letting $Y=\pi_{\mathcal{Q}, \infty}[\mathrm{ran}(f)]$ and $Z$ be the code of the $\pi$-pullback of the strategy for $\mathcal{N}$ determined by the strategy predicate of $\mathcal{Q}$, the empty condition of $\mathrm{Col}(\omega, {<}\delta)$ forces that
            \[
            \mathsf{CDM}^-\vert\gamma\models\phi[u, \zeta, Y, Z, x, \vec{\beta}].
            \]
        \end{itemize}

        Now let $\eta_0\in[\max\{\xi_Y, \pi_{\mathcal{P}, \mathcal{Q}}(\xi_Z)\}, \delta)$ such that $x, y, z\in\mathcal{Q}[h\uphar\eta_0]$. Let $\delta_0<\delta$ be the least Woodin cardinal of $\mathcal{Q}$ above $\eta_0$ and let $\eta_1\in(\delta_0, \delta)$ be an inaccessible cardinal of $\mathcal{Q}$ such that
        \[
        (*)\;\;\mathsf{CDM}^{-}\models w(\mathrm{Code}(\Sigma^h_{\mathcal{Q}\vert\eta_1}))>\sup f[\mathbb{R}^{h\uphar\delta_0}],
        \]
        where $w(\text{--})$ denotes the Wadge rank of a set of reals. Such an $\eta_1$ exists because $\cf(\Theta)>\omega$ in $\mathsf{CDM}^{-}$. Since $f$ is cofinal, there is an $r\in\mathbb{R}^*_h$ such that 
        \[
        f(r)>w(\mathrm{Code}(\Sigma^h_{\mathcal{Q}\vert\delta_1})),
        \]
        where $\delta_1<\delta$ is a sufficiently large Woodin cardinal of $\mathcal{Q}$ above $\eta_1$ such that $\mathrm{Code}(\Sigma^h_{\mathcal{Q}\vert\delta_1})$ is not projective in $\mathrm{Code}(\Sigma^h_{\mathcal{Q}\vert\eta_1})$. \footnote{Actually, one can chosen $\delta_1$ as the least Woodin cardinal of $\mathcal{Q}$ above $\eta_1$, see \cite[Lemma 3.4]{dm_of_self_it}.}

        Using the extender algebra at $\delta_0$, we can take an $\mathcal{Q}^*\in\mathcal{F}^*_g(\mathcal{Q})$ and an $\mathcal{Q}^*$-generic $h^*\subset\mathrm{Col}(\omega, {<}\delta)$ such that $\crit(\pi_{\mathcal{Q}, \mathcal{Q}^*})>\eta_0$, $h\uphar\eta_0\subset h^*$ and $r\in\mathcal{Q}^*[h^*\uphar\pi_{\mathcal{Q},\mathcal{Q}^*}(\delta_0)]$. Then let $\mathcal{R}$ be an genericity iterate of $\mathcal{Q}^*$ such that $\crit(\pi_{\mathcal{Q}^*, \mathcal{R}})>\pi_{\mathcal{Q},\mathcal{Q}^*}(\delta_0)$. Let $k\subset\mathrm{Col}(\omega, {<}\delta)$ be a maximal $\mathcal{R}$-generic such that $h^*\uphar\pi_{\mathcal{Q}, \mathcal{Q}^*}(\delta_0)\subset k$.
        
        Let $\pi^{+}_{\mathcal{V}_{\mathcal{Q}}, \mathcal{V}_{\mathcal{R}}}\colon\mathcal{V}_{\mathcal{Q}}[h\uphar\eta_0]\to\mathcal{V}_{\mathcal{R}}[h\uphar\eta_0]$ be the canonical liftup of $\pi_{\mathcal{V}_{\mathcal{Q}}, \mathcal{V}_{\mathcal{R}}}$.
        By \cref{stabilizing_parameters}, the elementarity of $\pi^{+}_{\mathcal{V}_{\mathcal{Q}}, \mathcal{V}_{\mathcal{R}}}\colon\mathcal{V}_{\mathcal{Q}}[x, y, z]\to\mathcal{V}_{\mathcal{R}}[x, y, z]$, which is the canonical liftup of $\pi_{\mathcal{V}_{\mathcal{Q}}, \mathcal{V}_{\mathcal{R}}}$, implies that
        \[
        \pi^+_{\mathcal{V}_{\mathcal{Q}}, \mathcal{V}_{\mathcal{R}}}(f)=\{\langle u, \zeta\rangle\in\mathbb{R}^*_k\times\Theta\mid\mathcal{V}_{\mathcal{R}}[x, y, z][u]\models\phi^*(u, \zeta, x, y, z, \delta, \vec{\beta}, \gamma)\}.
        \]
        Then the following observations imply $\pi^+_{\mathcal{V}_{\mathcal{Q}}, \mathcal{V}_{\mathcal{R}}}(f)=f$:
        \begin{enumerate}
            \item Because $\mathcal{R}$ is a genericity iterate of $\mathcal{Q}$,
            \[
            (\mathsf{CDM}^{-})^{\mathcal{V}_{\mathcal{Q}}[h]}=(\mathsf{CDM}^{-})^{\mathcal{V}_{\mathcal{R}}[k]}
            \]
            by \cref{pres_of_CDM}.
            \item Let $Y'=\pi_{\mathcal{R}, \infty}[\mathrm{ran}(f_y)]$, where $f_y\colon\omega\to\xi_Y$ is the function coded by $y$. Since $\crit(\pi_{\mathcal{Q}, \mathcal{R}})>\xi_Y$, \cref{pres_of_M_infty} implies that $Y'=\pi_{\mathcal{V}_{\mathcal{Q}}, \mathcal{V}_{\mathcal{R}}}(Y)$. Moreover, $\pi_{\mathcal{V}_{\mathcal{Q}}, \mathcal{V}_{\mathcal{R}}}(Y)=Y$ by \cref{stabilizing_parameters}. Therefore, $Y'=Y$.
            \item Since $\crit(\pi_{\mathcal{Q}, \mathcal{R}})>\pi_{\mathcal{P}, \mathcal{Q}}(\xi_Z)$, $\mathcal{Q}\vert\pi_{\mathcal{P}, \mathcal{Q}}(\xi_Z)=\mathcal{R}\vert\pi_{\mathcal{P}, \mathcal{Q}}(\xi_Z)$. So $Z$ is the code of the $\pi_{\mathcal{P}, \mathcal{R}}$-pullback of the strategy for $\mathcal{R}\vert\pi_{\mathcal{P}, \mathcal{Q}}(\xi_Z)$ determined by the strategy predicate of $\mathcal{V}_{\mathcal{R}}$.
        \end{enumerate}
        Now by the elementarity of $\pi^+_{\mathcal{V}_{\mathcal{Q}}, \mathcal{V}_{\mathcal{R}}}$, $(*)$ implies that
        \[
        \mathsf{CDM}^{-}\models w(\mathrm{Code}(\Sigma^k_{\mathcal{R}\vert\pi_{\mathcal{Q}, \mathcal{R}}(\eta_1)}))>\sup f[\mathbb{R}^{k\uphar\pi_{\mathcal{Q}, \mathcal{R}}(\delta_0)}].
        \]
        Since $r\in\mathbb{R}^{h^*\uphar\pi_{\mathcal{Q}, \mathcal{Q}^*}(\delta_0)}\subset\mathbb{R}^{k\uphar\pi_{\mathcal{Q}, \mathcal{R}}(\delta_0)}$, it follows that $w(\mathrm{Code}(\Sigma^k_{\mathcal{R}\vert\pi_{\mathcal{Q}, \mathcal{R}}(\eta_1)}))>f(r)$.
        As $\Sigma^k_{\mathcal{R}\vert\pi_{\mathcal{Q}, \mathcal{R}}(\eta_1)}$ is a tail strategy of $\Sigma^h_{\mathcal{Q}\vert\eta_1}$, $\mathrm{Code}(\Sigma^k_{\mathcal{R}\vert\pi_{\mathcal{Q}, \mathcal{R}}(\eta_1)})$ is projective in $\mathrm{Code}(\Sigma^h_{\mathcal{Q}\vert\eta_1})$.
        Then we have
        \[
        w(\mathrm{Code}(\Sigma^h_{\mathcal{Q}\vert\delta_1}))> w(\mathrm{Code}(\Sigma^k_{\mathcal{R}\vert\pi_{\mathcal{Q}, \mathcal{R}}(\eta_1)}))>f(r),
        \]
        which contradicts the choice of $r$.
    \end{proof}

    \begin{thm}\label{value_of_Theta}
        $\Theta^{\mathsf{CDM}^{-}}=\kappa_{\infty}$, where $\kappa_{\infty}=\pi_{\mathcal{P}, \infty}(\kappa)$.
    \end{thm}
    
    \begin{proof}
        Let $\Theta=\Theta^{\mathsf{CDM}^{-}}$. The next claim implies that $\Theta\leq\kappa_{\infty}$.
        Recall that $\alpha$ is a cutpoint of an lbr hod premouse $\mathcal{M}$ if there is no extender $E$ on the extender sequence of $\mathcal{M}$ such that $\crit(E)<\alpha\leq\lh(E)$.
        
        \begin{claim}
            $\Theta$ is a cutpoint of $\mathcal{M}_{\infty}$.
        \end{claim}
        \begin{proof}
            The claim follows from the proof of \cite[Theorem 1.7]{char_of_ext_in_HOD}, but we will write it down for the reader's convenience.
            
            We work in $\mathsf{CDM}^{-}$. Recall that $\mathsf{CDM}^{-}\models\mathsf{AD}_{\mathbb{R}}$. Suppose toward a contradiction that there is an extender $E$ on the extender sequence of $\mathcal{M}_{\infty}$ such that $\crit(E)<\Theta\leq\lh(E)$. Let $\kappa=\crit(E)$ and let $\theta_{\alpha}<\Theta$ be the least member of the Solovay sequence above $\kappa$. By \cite[Theorem 1.5]{char_of_ext_in_HOD}\footnote{The theorem is not stated in \cite{char_of_ext_in_HOD} in the generality we need, but see \cite[Theorem 0.3]{Suslin_and_cutpoints}.}, there is a countably complete ultrafilter $U$ over $\mathsf{CDM}^{-}$ such that $\kappa=\crit(U)$ and $\pi_{U}(\kappa)\geq\pi_{E}(\kappa)$. By Kunen's theorem (\cite[Theorem 7.6]{St09})\footnote{Some literature assumes $\mathsf{AD}+\mathsf{DC}$ for Kunen's theorem, but $\mathsf{AD}+\mathsf{DC}_{\mathbb{R}}$ is enough.}, $U$ is ordinal definable. Then there is an OD surjection $\power(\kappa)\to\pi_{U}(\kappa)$. Since $\theta_{\alpha+1}<\pi_{E}(\kappa)\leq\pi_{U}(\kappa)$, we can take an OD surjection $f\colon\power(\kappa)\to\theta_{\alpha+1}$. Let $A$ be any set of reals of Wadge rank $\theta_{\alpha}$. Then there is an $\mathrm{OD}(A)$ surjection $\mathbb{R}\to\kappa$. Moschovakis coding lemma (\cite[Section 7D]{Mo09}) implies that there is an $\mathrm{OD}(A)$ surjection $g\colon\mathbb{R}\to\power(\kappa)$. Then $f\circ g\colon\mathbb{R}\to\theta_{\alpha+1}$ is an $\mathrm{OD}(A)$ surjection, which is a contradiction.
        \end{proof}
        Suppose toward a contradiction that $\Theta<\kappa_{\infty}$. Then there is a $\mathcal{Q}\in\mathcal{F}^*_g(\mathcal{P})$ such that $\Theta=\pi_{\mathcal{Q}, \infty}(\eta)$ for some $\eta<\kappa_{\mathcal{Q}}$. Since $\kappa_{\mathcal{Q}}$ is the least $<\delta$-strong cardinal in $\mathcal{Q}$, it is a cutpoint of $\mathcal{Q}$. It follows that $\pi_{\mathcal{Q}, \infty}\uphar\kappa_{\mathcal{Q}}$ is an iteration map according to $\Sigma_{\mathcal{Q}}\uphar\mathcal{Q}\vert\kappa_{\mathcal{Q}}$. Since $\mathrm{Code}(\Sigma_{\mathcal{Q}}\uphar(\mathcal{Q}\vert\kappa_{\mathcal{Q}}))\in\Gamma^*_g$, there is a surjection from $\mathbb{R}^*_g$ onto $\Theta=\pi_{\mathcal{Q}, \infty}(\eta)$ is collapsed in $\mathsf{CDM}^{-}$, which is a contradiction.
        Therefore, $\Theta=\kappa_{\infty}$.
    \end{proof}
    
    \begin{thm}\label{HOD_in_CDM}
        In $\mathsf{CDM}^{-}$, $\HOD\|(\Theta^+)^{\HOD}=\mathcal{M}_{\infty}\vert(\kappa_{\infty}^+)^{\mathcal{M}_{\infty}}$.
    \end{thm}
    \begin{proof}
        In $\mathsf{CDM}^-$, the $\HOD$ analysis in \cite{CPMP} implies that
        \[
        \HOD\|\Theta=\mathcal{M}_{\infty}\vert\kappa_{\infty}
        \]
        and $\Sigma_{\mathcal{M}_{\infty}\vert\kappa_{\infty}}$ is ordinal definable.
        Moreover, $\mathcal{M}_{\infty}\vert(\kappa_{\infty}^+)^{\mathcal{M}_{\infty}}$ can be written as a stack of all sound lbr hod premice $\mathcal{M}$ such that $\mathcal{M}_{\infty}\vert\kappa_{\infty}\triangleleft\mathcal{M}, \rho(\mathcal{M})=\kappa_{\infty}$, and whenever $\pi\colon\mathcal{N}\to\mathcal{M}$ is elementary and $\mathcal{N}$ is countable, there is an $\omega_1$-iteration strategy $\Lambda$ for $\mathcal{N}$ such that $\Sigma^{\pi}_{\mathcal{M}_{\infty}\vert\kappa_{\infty}}\subset\Lambda$. It follows that $\mathcal{M}_{\infty}\vert(\kappa_{\infty}^+)^{\mathcal{M}_{\infty}}\subset\mathrm{HOD}$ in $\mathsf{CDM}^-$.
        
        To show that $\HOD\|(\Theta^+)^{\HOD}\subset\mathcal{M}_{\infty}$, let $A\subset\Theta$ be ordinal definable in $\mathsf{CDM}^{-}$.
        Take a formula $\phi$ and ordinal parameters $\vec{\beta}\in{}^{<\omega}\ord$ defining $A$, i.e.,
        \[
        A=\{\alpha<\Theta\mid\mathsf{CDM}^{-}\models\phi(\alpha, \vec{\beta})\}.
        \]
        Let $\mathcal{Q}$ be a genericity iterate of $\mathcal{P}$ such that $\vec{\beta}\in\mathrm{ran}(\pi_{\mathcal{Q}, \infty})$.
        Let $B\subset\kappa_{\mathcal{Q}}$ be such that
        \[
        B = \{\xi<\kappa_{\mathcal{Q}}\mid\emptyset\Vdash^{\mathcal{V}_{\mathcal{Q}}}_{\mathrm{Col}(\omega, {<}\delta)}\mathsf{CDM}^{-}\models\phi[\pi_{\mathcal{Q}, \infty}(\xi), \vec{\beta}]\}.
        \]
        Then $B\in\mathcal{Q}$.
        We want to show that $A=\pi_{\mathcal{Q}, \infty}(B)$.
        Let $\alpha\in\kappa_{\infty}$.
        Take a genericity iterate $\mathcal{R}$ of $\mathcal{Q}$ such that $\alpha\in\mathrm{ran}(\pi_{\mathcal{R}, \infty})$.
        By \cref{stabilizing_parameters}, $\pi_{\mathcal{V}_{\mathcal{Q}}, \mathcal{V}_{\mathcal{R}}}(\vec{\beta})=\vec{\beta}$ and thus
        \[
        \pi_{\mathcal{Q}, \mathcal{R}}(B)=\{\xi<\kappa_{\mathcal{R}}\mid\emptyset\Vdash^{\mathcal{V}_{\mathcal{R}}}_{\mathrm{Col}(\omega, {<}\delta)}\mathsf{CDM}^{-}\models\phi[\pi_{\mathcal{R}, \infty}(\xi), \vec{\beta}]\}
        \]
        Then
        \[
        \alpha\in A\iff \pi^{-1}_{\mathcal{R}, \infty}(\alpha)\in \pi_{\mathcal{Q}, \mathcal{R}}(B)\iff \alpha\in\pi_{\mathcal{Q}, \infty}(B).
        \]
        Therefore, $A=\pi_{\mathcal{Q}, \infty}(B)\in\mathcal{M}_{\infty}$.
    \end{proof}

    Now it remains to show (3) in \cref{main_thm_on_CDM}. The next lemma is not directly used for the proof of \cref{main_thm_on_CDM}, but we prove it to motivate later argument.

    \begin{prop}\label{char_of_nu_in_V[g]}
        Let $\Theta=\Theta^{\mathsf{CDM}^{-}}$. Let $\nu$ be a measure on $\kappa_{\infty}$ of Mitchell order 0 in $\mathcal{M}_{\infty}$. Let $A\subset\Theta$ in $\mathcal{M}_{\infty}$. Then in $\mathcal{V}[g]$, $A\in\nu$ if and only if $A$ contains a club subset of $\Theta\cap\mathrm{Cof}(\omega)$.
    \end{prop}
    \begin{proof}
        For any $\mathcal{Q}\in\mathcal{F}^*_{g}(\mathcal{P}, \Sigma)$, let $\kappa_{\mathcal{Q}}=\pi_{\mathcal{P}, \mathcal{Q}}(\kappa)$ and $\alpha_{\mathcal{Q}}=\sup\pi_{\mathcal{Q}, \infty}[\kappa_{\mathcal{Q}}]$.
        Now take $\mathcal{Q}\in\mathcal{F}^*_g(\mathcal{P}, \Sigma)$ such that $A\in\mathrm{ran}(\pi_{\mathcal{Q}, \infty})$. Let $\nu_{\mathcal{Q}}, A_{\mathcal{Q}}\in\mathcal{Q}$ be the preimages of $\nu, A$ under $\pi_{\mathcal{Q}, \infty}$, respectively.

        \begin{claim}\label{key_claim}
            $\pi_{\ult(\mathcal{Q}, \nu_{\mathcal{Q}}), \infty}(\kappa_{\mathcal{Q}})=\alpha_{\mathcal{Q}}$.
        \end{claim}
        \begin{proof}
            Since $\kappa_{\mathcal{Q}}$ is not measurable in $\ult(\mathcal{Q}, \nu_{\mathcal{Q}})$,
            \[
            \pi_{\ult(\mathcal{Q}, \nu_{\mathcal{Q}}), \infty}(\kappa_{\mathcal{Q}})=\sup\pi_{\ult(\mathcal{Q}, \nu_{\mathcal{Q}}), \infty}[\kappa_{\mathcal{Q}}]
            \]
            The coherency of $\nu_{\mathcal{Q}}$ and the positionality of $\Sigma$ implies that $\Sigma_{\ult(\mathcal{Q}, \nu_{\mathcal{Q}})\vert\kappa_{\mathcal{Q}}}=\Sigma_{\mathcal{Q}\vert\kappa_{\mathcal{Q}}}$. Also, $\kappa_{\mathcal{Q}}$ is a cutpoint of $\mathcal{Q}$ and $\ult(\mathcal{Q}, \nu_{\mathcal{Q}})$. It follows that for any $\xi<\kappa_{\mathcal{Q}}$,
            \[
            \pi_{\ult(\mathcal{Q}, \nu_{\mathcal{Q}}), \infty}(\xi)=\pi_{\mathcal{Q}, \infty}(\xi).
            \]
            Therefore, $\sup\pi_{\ult(\mathcal{Q}, \nu_{\mathcal{Q}}), \infty}[\kappa_{\mathcal{Q}}]=\alpha_{\mathcal{Q}}$, which completes the proof of the claim.
        \end{proof}
        
        Note that $\pi_{\mathcal{Q}, \infty}=\pi_{\ult(\mathcal{Q}, \nu_{\mathcal{Q}}), \infty}\circ \pi_{\nu_{\mathcal{Q}}}$, where $\pi_{\nu_{\mathcal{Q}}}$ is the ultrapower map associated with $\nu_{\mathcal{Q}}$. This equality and \cref{key_claim} imply that
        \begin{align*}
        A\in\nu &\iff A_{\mathcal{Q}}\in\nu_{\mathcal{Q}} \iff \kappa_{\mathcal{Q}}\in \pi_{\nu_{\mathcal{Q}}}(A_{\mathcal{Q}})\\
        & \iff\pi_{\ult(\mathcal{Q}, \nu_{\mathcal{Q}}), \infty}(\kappa_{\mathcal{Q}})\in A \iff \alpha_{\mathcal{Q}}\in A.
        \end{align*}
        Therefore, $\{\alpha_{\mathcal{Q}}\mid \mathcal{Q}\in\mathcal{F}^*_g(\mathcal{P})\land A\in\mathrm{ran}(\pi_{\mathcal{Q}, \infty})\}\subset A$.

        \begin{claim}\label{omega-club}
            The set $\{\alpha_{\mathcal{Q}}\mid \mathcal{Q}\in\mathcal{F}^*_g(\mathcal{P})\land A\in\mathrm{ran}(\pi_{\mathcal{Q}, \infty})\}$ contains a club subset of $\kappa_{\infty}\cap\mathrm{Cof}(\omega)$ in $\mathcal{V}[g]$.
        \end{claim}
        \begin{proof}
            Fix a sequence $\langle\beta_{\xi}\mid\xi<\delta\rangle$ that is cofinal in $\kappa_{\infty}$. For each $\xi<\delta$, take $\mathcal{Q}_{\xi}\in\mathcal{F}^*_g(\mathcal{P})$ such that $\beta_{\xi}\in\mathrm{ran}(\pi_{\mathcal{Q}_{\xi}, \infty})$, which implies $\beta_{\xi}<\alpha_{\mathcal{Q}_{\xi}}$. Let $\mathcal{R}_0 = \mathcal{Q}_0$. Inductively, for each $\xi<\delta$, we let $\mathcal{R}_{\xi+1}$ be the common iterate of $\mathcal{R}_{\xi}$ and $\mathcal{Q}_{\xi+1}$, and for each limit ordinal $\lambda<\delta$, let $\mathcal{R}_{\lambda}$ be the direct limit of $\langle\mathcal{R}_\xi\mid\xi<\lambda\rangle$. Then $\{\alpha_{\mathcal{R}_{\xi}}\mid\xi<\delta\}$ is a club subset of $\kappa_{\infty}\cap\mathrm{Cof}(\omega)$ in the given set.
        \end{proof}

        This completes the proof of \cref{char_of_nu_in_V[g]}.
    \end{proof}

    We need to show the equivalence in \cref{char_of_nu_in_V[g]} in $\mathsf{CDM}^{-}$, not $\mathcal{V}[g]$. The problem is that the iteration embeddings to $\mathcal{M}_{\infty}$ are not in $\mathsf{CDM}^{-}$, so the club set we found above is not in $\mathsf{CDM}^{-}$. To solve this issue, we make use of notion called (strongly) condensing sets, which was originally introduced in the context of core model induction by the fourth author. We first introduce several notations.

    \begin{itemize}
        \item For any non-dropping $\Sigma$-iterate $\mathcal{Q}$ of $\mathcal{P}$, we define $\mathcal{Q}^{\mathsf{b}}=\mathcal{Q}\vert(\pi_{\mathcal{P}, \mathcal{Q}}(\kappa)^+)^{\mathcal{Q}}$, which is called the \emph{bottom part of $\mathcal{Q}$}.
        \item For any $X\in\power_{\omega_1}(\mathcal{M}^{\mathsf{b}}_{\infty})$, let $\mathcal{Q}_{X}=\cHull^{\mathcal{M}_{\infty}^{\mathsf{b}}}(X)$ and let $\tau_X\colon\mathcal{Q}_X\to\mathcal{M}_{\infty}^{\mathsf{b}}$ be the uncollapse map. Also, let $\kappa_X = \tau_{X}^{-1}(\kappa_{\infty})$.
        \item For any $X\in\power_{\omega_1}(\mathcal{M}_{\infty}^{\mathsf{b}})$, let $\Psi_X$ be the $\tau_X$-pullback strategy of $\Sigma_{\mathcal{M}_{\infty}^{\mathsf{b}}}$.
        Also, let $\mathcal{M}_{\infty}(\mathcal{Q}_X, \Psi_X)$ be the direct limit of all countable $\Psi_X$-iterates of $\mathcal{Q}_X$ under iteration maps and let $\pi^{\Psi_X}_{\mathcal{Q}_X, \infty}\colon\mathcal{Q}_X\to\mathcal{M}_{\infty}(\mathcal{Q}_X, \Psi_X)$ be the associated direct limit map.
        \item For any $X\subset Y\in\power_{\omega_1}(\mathcal{M}^{\mathsf{b}}_{\infty})$, we define $\tau_{X, Y}\colon\mathcal{Q}_X\to\mathcal{Q}_Y$ by $\tau_{X, Y}=\tau_Y^{-1}\circ\tau_X$.
    \end{itemize}

    \begin{dfn}
        For any $X\subset Y\in\power_{\omega_1}(\mathcal{M}_{\infty}^{\mathsf{b}})$, We say that \emph{$Y$ extends $X$} if
        \[
        \mathcal{Q}_Y = \{\tau_{X, Y}(f)(s)\mid f\in \mathcal{Q}_X\land f\colon[\kappa_X]^{<\omega}\to\mathcal{Q}_X\land s\in[\kappa_Y]^{<\omega}\}.
        \]
    \end{dfn}
    
    \begin{dfn}
        Let $X\in\power_{\omega_1}(\mathcal{M}_{\infty}^{\mathsf{b}})$. We say that $X$ is \emph{condensing} if for any $Y$ extending $X$, there is a unique elementary map $k^X_Y\colon\mathcal{Q}_Y\to\mathcal{M}_{\infty}^{\mathsf{b}}$ such that
        \begin{enumerate}
            \item $\tau_X = k^X_Y \circ \tau_{X, Y}$ and
            \item $k^X_Y\uphar\kappa_Y=\pi^{\Psi_Y}_{\mathcal{Q}_Y, \infty}\uphar\kappa_Y$.
        \end{enumerate}
        We also say that $X$ is \emph{strongly condensing} if whenever $Y$ extends $X$, $Y$ is condensing.
    \end{dfn}

    Now we show \cref{char_of_nu_in_CDM} assuming \cref{strongly_condensing} on the existence of strongly condensing sets.
    We will show \cref{strongly_condensing} in the next section, because it is a general fact about $\mathcal{M}_{\infty}^{\mathsf{b}}$ and its proof is independent of the other arguments.
    
    \begin{thm}\label{char_of_nu_in_CDM}
        Let $\Theta=\Theta^{\mathsf{CDM}^{-}}$. Let $\nu$ be a measure on $\kappa_{\infty}$ of Mitchell order 0 in $\mathcal{M}_{\infty}$. Let $A\subset\Theta$ in $\mathcal{M}_{\infty}$. Then in $\mathsf{CDM}^{-}$, $A\in\nu$ if and only if $A$ contains a club subset of $\Theta\cap\mathrm{Cof}(\omega)$.
    \end{thm}
    \begin{proof}
        Let $A\subset\Theta$ in $\mathcal{M}_{\infty}$.
        Then there is a genericity iterate $\mathcal{Q}$ of $\mathcal{P}$ such that $A\in\mathrm{ran}(\pi_{\mathcal{Q}, \infty})$.
        By \cref{strongly_condensing}, $X:=\pi_{\mathcal{Q}, \infty}[\mathcal{Q}^{\mathsf{b}}]$ is strongly condensing.
        Take any $Y\in\power_{\omega_1}(\mathcal{M}_{\infty}^{\mathsf{b}})$ such that $Y$ extends $X$. For any such $Y$, let $\nu_Y = \tau_Y^{-1}(\nu)$ and $\alpha_Y=\sup\tau_Y[\kappa_Y]=\sup(\Theta\cap\mathrm{Hull}^{\mathcal{M}_{\infty}^{\mathsf{b}}}(Y))$.
        
        \begin{claim}\label{key_claim_2}
            There is a $Z\in\power_{\omega_1}(\mathcal{M}_{\infty}^{\mathsf{b}})$ such that  $\mathcal{Q}_Z=\ult(\mathcal{Q}_Y, \nu_Y)$ and $\tau_{Y, Z}=\pi_{\nu_Y}\colon\mathcal{Q}_Y\to\ult(\mathcal{Q}_Y, \nu_Y)$, which implies that $Z$ extends $Y$. Moreover, for such a $Z$, letting $k^Y_Z\colon\mathcal{Q}_Z\to\mathcal{M}_{\infty}^{\mathsf{b}}$ be the elementary map witnessing that $Y$ is condensing, $k^Y_Z(\kappa_Y)=\alpha_Y$.
        \end{claim}
        \begin{proof}
            Note that $\nu$ is countably complete by \cref{char_of_nu_in_V[g]}. Then we can take a factor map $\sigma\colon\ult(\mathcal{Q}_Y, \nu_Y)\to\mathcal{M}_{\infty}$ such that $\pi_Y=\sigma\circ\pi_{\nu_Y}$. Let $Z=\mathrm{ran}(\sigma)$. Then $Z$ clearly satisfies the desired property.
            
            Now notice that $\pi^{\Psi_Z}_{\mathcal{Q}_Z, \infty}$ is continuous at $\kappa_Y$ because $\kappa_Y$ is not measurable in $\mathcal{Q}_Z=\ult(\mathcal{Q}_Y, \nu_Y)$. Also, $k^Y_Z\uphar\kappa_Z=\pi^{\Psi_Z}_{\mathcal{Q}_Z, \infty}\uphar\kappa_Z$ as $Y$ is condensing. Then it follows from $\kappa_Z=\pi_{\nu_Y}(\kappa_Y)>\kappa_Y$ that
            \[
            k^Y_Z(\kappa_Y)=\sup k^Y_Z[\kappa_Y].
            \]
            Moreover, since $\tau_Y= k^Y_Z\circ\tau_{Y, Z}$ and $\tau_{Y, Z}=\pi_{\nu_Y}$ has critical point $\kappa_Y$, we have
            \[
            \sup k^Y_Z [\kappa_Y] = \sup\tau_Y[\kappa_Y],
            \]
            which completes the proof of the claim.
        \end{proof}
        
        It follows from \cref{key_claim_2} that
        \begin{align*}
            A\in\nu &\iff \tau_Y^{-1}(A)\in\nu_Y \iff \kappa_Y\in \pi_{\nu_Y}(\tau_Y^{-1}(A))\\
            & \iff k^Z_Y(\kappa_Y)\in A \iff \alpha_Y\in A.
        \end{align*}
        Therefore, $\{\alpha_Y\mid Y\in\power_{\omega_1}(\mathcal{M}_{\infty}^{\mathsf{b}})\land\text{$Y$ extends $X$}\}\subset A$.

        \begin{claim}\label{omega-club_2}
            The set $\{\alpha_Y\mid Y\in\power_{\omega_1}(\mathcal{M}_{\infty}^{\mathsf{b}})\land\text{ $Y$ extends $X$}\}$ contains a club subset of $\Theta\cap\mathrm{Cof}(\omega)$ in $\mathsf{CDM}^{-}$.
        \end{claim}
        \begin{proof}
            For $\eta<\Theta$, let
            \[
            f(\eta)=\sup\{\alpha<\Theta\mid \alpha\in\mathrm{Hull}^{\mathcal{M}_{\infty}^{\mathsf{b}}}(\eta\cup X)\}.
            \]
            Let $C=\{\eta<\Theta\mid f[\eta]\subset\eta\land\cf(\eta)=\omega\}$. $C$ is clearly a club subset of $\Theta\cap\mathrm{Cof}(\omega)$ in $\mathsf{CDM}^{-}$. We shall show that for any $\eta\in C$, $\eta$ is $\alpha_Y$ for some $Y\in\power(\mathcal{M}_{\infty}^{\mathsf{b}})$ such that $Y$ extends $X$. Now fix $\eta\in C$ and take a countable cofinal subset $D\subset\eta$. Then let $Y^*:=D\cup X$. Clearly $\eta\leq\sup(\Theta\cap\mathrm{Hull}^{\mathcal{M}_{\infty}^{\mathsf{b}}}(Y^*))=\alpha_{Y^*}$. On the other hand, if $\alpha<\alpha_{Y^*}$, then $\alpha\in\mathrm{Hull}^{\mathcal{M}_{\infty}^{\mathsf{b}}}(\beta\cup X)$ for some $\beta\in D$ and thus $\alpha\leq f(\beta)<\eta$ by the choice of $\eta$. Hence, $\eta\leq\alpha_{Y^*}$. We have just shown that $\eta=\alpha_{Y^*}$, so it suffices to find $Y$ extending $X$ such that $\alpha_{Y}=\alpha_{Y^*}$. Let $E$ be the extender of length $\kappa_{Y^*}$ induced by $\tau_{X, Y^*}$. Then let $\sigma\colon\ult(\mathcal{Q}_X, E)\to\mathcal{Q}_{Y^*}$ be the factor map and set $Y=\mathrm{ran}(\tau_{Y^*}\circ\sigma)$. Then $Y$ extends $X$. As $\crit(\sigma)\geq\kappa_{Y^*}$, $\kappa_Y=\kappa_{Y^*}$ and $\alpha_Y = \alpha_{Y^*}$.
        \end{proof}

        This completes the proof of \cref{char_of_nu_in_CDM} except for showing \cref{strongly_condensing}.
    \end{proof}

    \subsection{Existence of condensing sets}

    Now our goal is to show \cref{strongly_condensing}.
    To give an easy but useful characterization of condensing sets, we introduce one more technical notion.

    \begin{dfn}
        Let $X\in\power_{\omega_1}(\mathcal{M}_{\infty}^{\mathsf{b}})$ and let $A\in\mathcal{M}_{\infty}^{\mathsf{b}}$. Then we write
        \begin{align*}
            T_{X, A} &= \{\langle\phi, s\rangle\mid s\in[\kappa_X]^{<\omega}\land\mathcal{M}_{\infty}^{\mathsf{b}}\models\phi[\tau_X(s), A]\}\\
            T^*_{X, A} &= \{\langle\phi, s\rangle\mid s\in[\kappa_X]^{<\omega}\land\mathcal{M}_{\infty}^{\mathsf{b}}\models\phi[\pi^{\Psi_X}_{\mathcal{Q}_X, \infty}(s), A]\}
        \end{align*}
        We say that \emph{$X$ has $A$-condensation} if whenever $Y$ extends $X$, $T_{Y, A}=T^*_{Y, A}$.
    \end{dfn}

    \begin{rem}
        Note that if $A\in\mathrm{ran}(\tau_X)$, then
        \[
        T_{X, A} = \{\langle\phi, s\rangle\mid s\in[\kappa_X]^{<\omega}\land\mathcal{Q}_X\models\phi[s, \tau_X^{-1}(A)]\}
        \]
        and thus $T_{X, A}\in\mathcal{Q}_X$.
        So, if $X\subset Y$ then $\tau_{X, Y}(T_{X, A})=T_{Y, A}$ by the elementarity of $\tau_{X, Y}$.
        On the other hand, we do not know $T^*_{X, A}\in\mathcal{Q}_X$ a priori.
        Even if $T^*_{X, A}\in\mathcal{Q}_X$, the elementarity of $\tau_{X, Y}$ does not imply $\tau_{X, Y}(T^*_{X, A})= T^*_{Y, A}$.
    \end{rem}

    \begin{lem}\label{char_of_condensing}
        Let $X\in\power_{\omega_1}(\mathcal{M}_{\infty}^{\mathsf{b}})$.
        Then $X$ is condensing if and only if $X$ has $A$-condensation for any $A\in X$.
    \end{lem}
    \begin{proof}
        Suppose that $X$ is condensing.
        Let $Y\in\power_{\omega_1}(\mathcal{M}_{\infty}^{\mathsf{b}})$ extending $X$ and let $k^X_Y\colon\mathcal{Q}_Y\to\mathcal{M}_{\infty}^{\mathsf{b}}$ be the unique elementary map witnessing it.
        Also, let $A\in X$ and write $A_X = \tau_X^{-1}(A)$.
        Then for any formula $\phi$ and $s\in[\kappa_X]^{<\omega}$,
        \begin{align*}
        \langle\phi, s\rangle\in T_{Y, A} &\iff \mathcal{M}_{\infty}^{\mathsf{b}}\models\phi[\tau_Y(s), \tau_X(A_X)]\\
        &\iff \mathcal{Q}_Y\models[s, \tau_{X, Y}(A_X)]\;\;\text{by the elementarity of $\tau_Y$}\\
        &\iff \mathcal{M}_{\infty}^{\mathsf{b}}\models\phi[k^X_Y(s), (k^X_Y\circ\tau_{X, Y})(A_X)]\;\;\text{by the elementarity of $k^X_Y$}\\
        &\iff \mathcal{M}_{\infty}^{\mathsf{b}}\models\phi[\pi^{\Psi_Y}_{\mathcal{Q}_Y, \infty}(s), \tau_X(A_X)]\;\;\text{by the property of $k^X_Y$}\\
        &\iff \langle\phi, s\rangle\in T^*_{Y, A}.
        \end{align*}
        Therefore, $X$ has $A$-condensation.
        
        To show the reverse direction, suppose that $X$ has $A$-condensation for any $A\in X$.
        It easily follows that $X$ has $A$-condensation for any $A\in\mathrm{ran}(\tau_X)$.
        Let $Y\in\power_{\omega_1}(\mathcal{M}_{\infty}^{\mathsf{b}})$ extending $X$. We define $k^X_Y\colon\mathcal{Q}_Y\to\mathcal{M}_{\infty}^{\mathsf{b}}$ by
        \[
            k^X_Y(\tau_{X, Y}(f)(s))=\tau_X(f)(\pi^{\Psi_Y}_{\mathcal{Q}_Y, \infty}(s))
        \]
        for any $f\in\mathcal{Q}_X$ and $s\in[\kappa_Y]^{<\omega}$. This is a well-defined elementary embedding; for any formula $\phi$ and $s\in[\kappa_Y]^{<\omega}$,
        \begin{align*}
            \mathcal{Q}_Y\models\phi[\tau_{X, Y}(f)(s)] & \iff \mathcal{M}_{\infty}^{\mathsf{b}}\models\phi[\tau_X(f)(\tau_Y(s))]\\
            & \iff \langle\phi^*, s\rangle \in T_{Y, \tau_X(f)}=T^*_{Y, \tau_X(f)}\\
            & \iff \mathcal{M}_{\infty}^{\mathsf{b}}\models\phi[\tau_X(f)(\pi^{\Psi_Y}_{\mathcal{Q}_Y, \infty}(s))],
        \end{align*}
        where $\phi^*$ is a formula such that $\phi^*[\tau_Y(s), \tau_X(f)]\equiv\phi[\tau_X(f)(\tau_Y(s))]$.
    \end{proof}

    We are ready to prove the following key theorem on condensing sets.
    
    \begin{thm}\label{condensing}
        $\pi_{\mathcal{P}, \infty}[\mathcal{P}^{\mathsf{b}}]$ is condensing.
    \end{thm}
    \begin{proof}
        Suppose toward a contradiction that $\pi_{\mathcal{P}, \infty}[\mathcal{P}^{\mathsf{b}}]$ is not condensing.
        Then by \cref{char_of_condensing}, we can take $A\in\pi_{\mathcal{P}, \infty}[\mathcal{P}^{\mathsf{b}}]$ such that $\pi_{\mathcal{P}, \infty}[\mathcal{P}^{\mathsf{b}}]$ does not have $A$-condensation. We will inductively construct a sequence $\langle \mathcal{P}_i, X_i, Y_i\mid i<\omega\rangle$ such that for all $i<\omega$,
        \begin{enumerate}
            \item $\mathcal{P}_i \in \mathcal{F}^*_g(\mathcal{P})$,
            \item $X_i\subset Y_i\subset X_{i+1}\in\power_{\omega_1}(\mathcal{M}_{\infty}^{\mathsf{b}})$,
            \item $X_i=\pi_{\mathcal{P}_i, \infty}[\mathcal{P}_i^{\mathsf{b}}]$,
            \item $X_i$ does not have $A$-condensation,
            \item $Y_i$ extends $X_i$ and $T_{Y_i, A}\neq T^*_{Y_i, A}$.
        \end{enumerate}
        First, let $\mathcal{P}_0=\mathcal{P}, X_0=\pi_{\mathcal{P}, \infty}[\mathcal{P}^{\mathsf{b}}]$. Also, choose $Y_0$ extending $X_0$ such that $T_{Y_0, A}\neq T^*_{Y_0, A}$. Clearly $\mathcal{P}_0, X_0, Y_0$ satisfy the above conditions. Next suppose that we have defined $\mathcal{P}_i, X_i, Y_i$ for some $i<\omega$. Then since $Y_i$ is a countable subset of $\mathcal{M}_{\infty}^{\mathsf{b}}$, one can find $\mathcal{P}_{i+1}\in\mathcal{F}^*_g(\mathcal{P}_i)$ such that $Y_i\subset\mathrm{ran}(\pi_{\mathcal{P}_{i+1}, \infty})$.
        Let $X_{i+1}=\pi_{\mathcal{P}_{i+1}, \infty}[\mathcal{P}_{i+1}^{\mathsf{b}}]$.
        Then $\mathcal{P}_{i+1}$ and $X_{i+1}$ satisfy the conditions (1)--(3).
        \begin{claim}
            $X_{i+1}$ does not have $A$-condensation.
        \end{claim}
        \begin{proof}
            Let $\mathcal{R}$ be a genericity iteration of $\mathcal{P}$ above $\kappa_{X_{i+1}}$. Let $A_{\mathcal{P}}=\pi^{-1}_{\mathcal{P}, \infty}(A)$. Then we have
            \begin{align*}
                \pi_{\mathcal{V}_{\mathcal{P}}, \mathcal{V}_{\mathcal{R}}}(A) &= \pi_{\mathcal{V}_{\mathcal{P}}, \mathcal{V}_{\mathcal{R}}}(\pi_{\mathcal{P}, \infty}(A_{\mathcal{P}}))\\
                & =\pi_{\mathcal{R}, \infty}(\pi_{\mathcal{P}, \mathcal{R}}(A_{\mathcal{P}}))\\
                & =\pi_{\mathcal{P}, \infty}(A_{\mathcal{P}})=A.
            \end{align*}
            Note that this calculation is the same as in \cref{stabilizing_parameters}. Then by the elementarity of $\pi_{\mathcal{V}_{\mathcal{P}}, \mathcal{V}_{\mathcal{R}}}$, $\pi_{\mathcal{R}, \infty}[\mathcal{R}^{\mathsf{b}}]$ does not have $A$-condensation. Since $\pi_{\mathcal{P}_{i+1}, \mathcal{R}}$ does not change the bottom part, $X_{i+1}=\pi_{\mathcal{R}, \infty}[\mathcal{R}^{\mathsf{b}}]$.
        \end{proof}
        Now we can take $Y_{i+1}$ extending $X_{i+1}$ such that $T_{Y_{i+1}, A}\neq T^*_{Y_{i+1}, A}$, which completes our inductive construction. Let $\mathcal{Q}_i=\mathcal{Q}_{Y_i}$. We should give shorter names to the maps we have:
        \begin{itemize}
            \item $\pi_{i, i+1} := \pi_{\mathcal{P}_i, \mathcal{P}_{i+1}}\colon\mathcal{P}_i\to\mathcal{P}_{i+1}$.
            \item $\pi_{i, i+1}^{\mathsf{b}}:=\pi_{i, i+1}\uphar\mathcal{P}_i^{\mathsf{b}}=\tau_{X_i, X_{i+1}}\colon\mathcal{P}_i^{\mathsf{b}}\to\mathcal{P}_{i+1}^{\mathsf{b}}$
            \item $\tau_i := \tau_{X_i, Y_{i+1}}\colon\mathcal{P}_i^{\mathsf{b}}\to\mathcal{Q}_i$.
            \item $\sigma_i := \tau_{X_i, Y_{i+1}}\colon\mathcal{Q}_i \to \mathcal{P}_{i+1}^{\mathsf{b}}$.
        \end{itemize}
        See \cref{fig1}.
        
        \begin{figure}[h]
            \centering
            \begin{tikzcd}
                \mathcal{P}_i^{\mathsf{b}} \arrow[rr, "{\pi_{i, i+1}^{\mathsf{b}}}"] \arrow[rd, "{\tau_i}"'] & & (\mathcal{P}_{i+1}^*)^{\mathsf{b}} \arrow[rr, Rightarrow, no head] &  & \mathcal{P}_{i+1}^{\mathsf{b}} \\
                & \mathcal{Q}_i \arrow[rrru, "{\sigma_i}"']  &  &  & 
            \end{tikzcd}
            \caption{Maps between $\mathcal{P}_i^{\mathsf{b}}$'s and $\mathcal{Q}_i$'s.}\label{fig1}
        \end{figure}
        
        The next step is lifting up our commutative diagram. Let $\mathcal{Q}_i^+ = \ult(\mathcal{P}_i, E_i)$, where $E_i$ is the extender of length $\kappa_{Y_i}$ derived from $\tau_{X_i, Y_i}$. Let $\tau_i^+\colon\mathcal{P}_i\to\mathcal{Q}_i^+$ be the ultrapower map. It is easy to see that $(\mathcal{Q}_i^+)^{\mathsf{b}}=\mathcal{Q}_i$ and $\tau_{X_i, Y_i}^+\uphar\mathcal{Q}_i = \tau_{X_i, Y_i}$. Let $\mathcal{P}^*_{i+1}$ be the earliest model in the iteration tree $\mathcal{T}_{\mathcal{P}_i, \mathcal{P}_{i+1}}$ such that $(\mathcal{P}_{i+1}^*)^{\mathsf{b}}=(\mathcal{P}_{i+1})^{\mathsf{b}}$. Note that $\mathcal{P}^*_{i+1}$ is on the main branch. Let $\pi^0_{i, i+1}=\pi_{\mathcal{P}_i, \mathcal{P}_{i+1}^*}$ and let $\pi^1_{i, i+1}=\pi_{\mathcal{P}^*_{i+1}, \mathcal{P}_{i+1}}$. Then $\pi_{i, i+1}=\pi^1_{i, i+1}\circ\pi^0_{i, i+1}$. It is not hard to see that $\mathcal{P}_{i+1}^* = \ult(\mathcal{Q}_i^+, F_i)$, where $F_i$ is the extender of length $\kappa_{X_{i+1}}$ derived from $\sigma_i$. Let $\sigma^*_i\colon\mathcal{Q}_i^+\to\mathcal{P}_{i+1}^*$ be the ultrapower map. We define $\sigma_i^+\colon\mathcal{Q}_i^+\to\mathcal{P}_{i+1}$ by $\sigma_i^+ =\pi^1_{i, i+1}\circ\sigma^*_i$. Finally, let $\Psi_i$ be the $\sigma_i^+$-pullback strategy of $\Sigma_{\mathcal{P}_{i+1}}$. See \cref{fig2}.

        \begin{figure}[h]
            \centering
            \begin{tikzcd}
                \mathcal{P}_i \arrow[rr, "{\pi^0_{i, i+1}}"] \arrow[rd, "{\tau_i^+}"'] \arrow[rrrr, "{\pi_{i, i+1}}", bend left] &  & \mathcal{P}_{i+1}^* \arrow[rr, "{\pi^1_{i, i+1}}"]  &  & \mathcal{P}_{i+1}\\
                & \mathcal{Q}_i^+ \arrow[ru, "{\sigma_i^*}"] \arrow[rrru, "{\sigma_i^+}"'] &  &  & 
            \end{tikzcd}
            \caption{Lifted maps between $\mathcal{P}_i$'s and $\mathcal{Q}_i^+$'s.}\label{fig2}
        \end{figure}
        
        The third step is the simultaneous genericity iteration to make all reals in $\mathcal{P}[g]$ generic using $\Sigma_{\mathcal{P}_i}$'s and $\Psi_i$'s above the bottom parts.
        Let $\langle x_{\alpha}\mid\alpha<\delta\rangle$ be an enumeration of $\mathbb{R}^{\mathcal{P}[g]}$ and let $\langle w_{\alpha}\mid\alpha<\delta\rangle$ be an increasing sequence of windows of $\mathcal{P}$ such that $\inf(w_0)\geq\kappa$.
        Then we will do similar construction as in the proof of \cite[Theorem 6.28]{HOD_as_a_core_model}.
        We sketch the construction:
        \begin{enumerate}
            \item Take a $\Sigma_{\mathcal{P}_0}$-iterate $\mathcal{P}_{0, 1}$ of $\mathcal{P}_0$ making $x_0$ generic using the extender algebra based on $w_0$.
            \item Then let $\mathcal{Q}^*_{0, 1}$ be the last model of $\pi_{0,1}\mathcal{T}_{\mathcal{P}_0, \mathcal{P}_{0, 1}}$, which is according to $\Psi_0$, and let $\tau_{0, 1}^*\colon\mathcal{P}_{0, 1}\to\mathcal{Q}_{0, 1}^*$ be the copy map.
            \item Take a $(\Psi_0)_{\mathcal{Q}_{0, 1}^*, \pi_{0,1}\mathcal{T}_{\mathcal{P}_0, \mathcal{P}_{0, 1}}}$-iterate $\mathcal{Q}_{0, 1}^+$ of $\mathcal{Q}_{0, 1}^*$ making $x_0$ generic using the extender algebra based on $\tau_{0, 1}^*(w_0)$.
            \item Let $\tau^+_{0, 1}\colon\mathcal{P}_{0, 1}\to\mathcal{Q}^+_{0, 1}$ be defined by $\tau^+_{0, 1}=\pi_{\mathcal{Q}^*_{0, 1}, \mathcal{Q}^+_{0, 1}}\circ\tau^*_{0, 1}$.
        \end{enumerate}
        Repeating such construction, we can also define a $\Sigma_{\mathcal{P}_1}$-iterate $\mathcal{P}_{1, 1}$ of $\mathcal{P}_1$ and an elementary map $\sigma^+_{0, 1}\colon\mathcal{Q}^+_{0, 1}\to\mathcal{P}_{1, 1}$ such that $\sigma^+_{0, 1}\circ\pi_{\mathcal{Q}^+_0, \mathcal{Q}^+_{0, 1}}=\pi_{\mathcal{P}_1, \mathcal{P}_{1, 1}}\circ\sigma^+_0$. Furthermore, we can inductively define $\mathcal{P}_{i, \alpha}, \mathcal{Q}^+_{i, \alpha}, \tau^+_{i, \alpha}\colon\mathcal{P}_{i, \alpha}\to\mathcal{Q}^+_{i, \alpha+1}, \sigma^+_{i, \alpha}\colon\mathcal{Q}^+_{i, \alpha}\to\mathcal{P}_{i+1, \alpha}$ for $i<\omega$ and $\alpha<\delta$ such that for any $i<\omega$ and any $\alpha<\delta$,
        \begin{itemize}
            \item $x_{\alpha}$ is generic over $\mathcal{P}_{i, \alpha}$ and $\mathcal{Q}^+_{i, \alpha}$ via the extender algebra based on the image of $w_{\alpha}$.
            \item If $\beta<\alpha$, then
            \begin{align*}
                \tau^+_{i, \alpha}\circ\pi_{\mathcal{P}_{i, \beta}, \mathcal{P}_{i, \alpha}} &= \pi_{\mathcal{Q}^+_{i, \beta}, \mathcal{Q}^+_{i, \alpha}}\circ\tau^+_{i, \beta},\\
                \sigma^+_{i, \alpha}\circ\pi_{\mathcal{Q}^+_{i, \beta}, \mathcal{Q}^+_{i, \alpha}} &= \pi_{\mathcal{P}_{i+1, \beta}, \mathcal{P}_{i+1, \alpha}}\circ\sigma^+_{i, \beta}.
            \end{align*}
        \end{itemize}
        Finally, let $\mathcal{P}_{i, \delta}$ and $\mathcal{Q}^+_{i, \delta}$ for each $i<\omega$ as the direct limit of $\mathcal{P}_{i, \alpha}$'s and $\mathcal{Q}^+_{i, \alpha}$'s respectively. Then they are genericity iterates of $\mathcal{P}_i$ and $\mathcal{Q}^+_i$ respectively and their derived models are all equal to the derived models computed in $\mathcal{P}[g]$ by \cref{pres_of_CDM}. Also, the direct limit $\mathcal{P}_{\infty}$ of all $\mathcal{P}_{i, \delta}$'s and $\mathcal{Q}^+_{i, \delta}$'s is also well-founded because it can be embedded into $\mathcal{M}_{\infty}$. See \cref{fig3}

        \begin{figure}[h]
            \begin{tikzcd}
                \mathcal{P}_i \arrow[dd] \arrow[rr, "\tau^+_i"]                         &  & \mathcal{Q}^+_i \arrow[dd] \arrow[rr, "\sigma^+_i"]                         &  & \mathcal{P}_{i+1} \arrow[dd]           &  & \\
                & & & & & & \\
                {\mathcal{P}_{i, \beta}} \arrow[rr, "{\tau^+_{i, \beta}}"] \arrow[d]    &  & {\mathcal{Q}^+_{i, \beta}} \arrow[rr, "{\sigma^+_{i, \beta}}"] \arrow[d]    &  & {\mathcal{P}_{i+1, \beta}} \arrow[d]   &  & \\
                {\mathcal{P}_{i, \alpha}} \arrow[rr, "{\tau^+_{i, \alpha}}"] \arrow[dd] &  & {\mathcal{Q}^+_{i, \alpha}} \arrow[rr, "{\sigma^+_{i, \alpha}}"] \arrow[dd] &  & {\mathcal{P}_{i+1, \alpha}} \arrow[dd] &  & \\
                & & & & & &\\
                {\mathcal{P}_{i, \delta}} \arrow[rr, "{\tau^+_{i, \delta}}"]            &  & {\mathcal{Q}^+_{i, \delta}} \arrow[rr, "{\sigma^+_{i, \delta}}"]            &  & {\mathcal{P}_{i+1, \delta}} \arrow[rr] &  & \mathcal{P}_{\infty}
            \end{tikzcd}
            \caption{Simultaneous genericity iteration.}\label{fig3}
        \end{figure}

        Now we are ready to obtain a contradiction.
        First, $T_{X_i, A}=T^*_{X_i, A}$ for all $i<\omega$ simply because $\tau_i = \pi_{\mathcal{P}_i, \infty}\uphar\mathcal{P}_i^{\mathsf{b}}$.
        There is a formula $\theta$ that defines $T^*_{X_i, A}$ from ordinal parameters $t$ in the derived model of $\mathcal{P}_{i, \delta}$'s at $\delta$.
        By the elementarity of $\tau^+_{i, \delta}$, $\langle\phi, s\rangle\in T_{\mathcal{Q}_i, A}$ if and only if the derived model of $\mathcal{Q}_{i, \delta}^+$ at $\delta$ satisfies $\theta(\langle\phi, s\rangle, \tau^+_{i, \delta}(t))$.
        Since $\mathcal{P}_{\infty}$ is well-founded, there is an $n<\omega$ such that for any $i\geq n$, $\tau^+_{i, \delta}(t)=t$.
        Because the derived models of $\mathcal{P}_{i, \delta}$'s and $\mathcal{Q}^+_{i, \delta}$'s at $\delta$ are all the same as the derived models of $\mathcal{P}$ at $\delta$, the derived model of $\mathcal{Q}_{i, \delta}^+$ at $\delta$ satisfies $\theta(\langle\phi, s\rangle, t)$ if and only if $\mathcal{M}_{\infty}^{\mathsf{b}}\models\phi[\pi^{\Psi_i}_{\mathcal{Q}_i, \infty}(s), A]$.
        These arguments imply that for any $i\geq n$, $T_{Y_i, A}=T^*_{Y_i, A}$, which contradicts the choice of $Y_i$'s.
    \end{proof}

    A small modification of the last proof gives us strong condensation.
    
    \begin{thm}\label{strongly_condensing}
        $\pi_{\mathcal{P}, \infty}[\mathcal{P}^{\mathsf{b}}]$ is strongly condensing.
        Moreover, for any genericity iterate $\mathcal{Q}$ of $\mathcal{P}$, $\pi_{\mathcal{Q}, \infty}[\mathcal{Q}^{\mathsf{b}}]$ is strongly condensing.
    \end{thm}
    \begin{proof}
        Suppose that $\pi_{\mathcal{P}, \infty}[\mathcal{P}^{\mathsf{b}}]$ is not strongly condensing.
        Then there is a $Y^*$ extending $\pi_{\mathcal{P}, \infty}[\mathcal{P}^{\mathsf{b}}]$ that is not condensing.
        Take $A\in Y^*$ such that $Y^*$ does not have $A$-condensation.
        Let $\mathcal{P}_0=\mathcal{P}, X_0 = \pi_{\mathcal{P}, \infty}[\mathcal{P}^{\mathsf{b}}], Y_0^* = Y^*$.
        Also, let $Y_0$ extending $Y_0^*$ such that $T_{Y_0, A}\neq T^*_{Y_0, A}$.
        Now we can inductively construct $\mathcal{P}_i, X_i, Y_i$ for $i>0$ with the same property as before.
        The key claim is that for each $i<\omega$, $X_i$ has an extension $Y_i$ that does not have $A$-condensation, which can be shown by the same proof.
        Therefore, the proof of \cref{condensing} leads us to a contradiction.
        The moreover part of the lemma also follows from the same argument.
    \end{proof}
 
 
	\section{Forcing argument}

    We devote this section to the proof of \cref{MainTheorem}. As we mentioned in the paragraph right after \cref{main_thm_on_CDM}, it is consistent relative to a Woodin limit of Woodin cardinals that there is a hod pair $(\mathcal{V}, \Omega)$ together with a regular limit of Woodin cardinals $\delta$ satisfying the assumption of \cref{main_thm_on_CDM}. Let $g\subset\mathrm{Col}(\omega, {<}\delta)$ be $\mathcal{V}$-generic. In this section, we write
    \[
    W=(\mathsf{CDM}^-)^{\mathcal{V}[g]}.
    \]
    It is enough to show the following.
	
	\begin{thm}\label{main_thm_on_Pmax}
         Let $G*H\subset (\mathbb{P}_{\mathrm{max}}*\mathrm{Add}(\Theta, 1))^W$ be $W$-generic.\footnote{Here, $\mathrm{Add}(\gamma, 1)$ is a forcing poset to add a Cohen subset of $\gamma$.} Then in $W[G*H]$, $\mathsf{ZFC}$ holds and for any $\kappa\in\{\omega_1, \omega_2, \omega_3\}$, the restriction of the club filter on $\kappa\cap\mathrm{Cof}(\omega)$ to $\HOD$ is an ultrafilter in $\HOD$.
	\end{thm}

	Note that $W$ does not satisfy $\mathsf{AC}$, but has the desired property for $\omega_1, \omega_2$ and $\Theta$ by \cref{main_thm_on_CDM}. We force over $W$ with $(\mathbb{P}_{\mathrm{max}}*\mathrm{Add}(\Theta,1))^W$ to collapse $\Theta$ to be $\omega_3$ and obtain a $\mathsf{ZFC}$ model. Then we argue that $\HOD^W=\HOD^{W[G*H]}$ to ensure that $W[G*H]$ is a desired model.
    We freely use the standard facts of the $\mathbb{P}_{\mathrm{max}}$ forcing written in \cite{Larson2010forcing}.
    \begin{lem}\label{another_representation_of_W}
        For any cofinal $X\in\power_{\omega_1}(\mathcal{M}_{\infty}^{\mathsf{b}})$, $W=L(\mathcal{M}_{\infty}^{\mathsf{b}}, X, \Gamma^*_g, \mathbb{R}^*_g)$.
    \end{lem}
    \begin{proof}
        We write $\Theta=\Theta^{W}$.
        Let $X\in\power_{\omega_1}(\mathcal{M}_{\infty}^{\mathsf{b}})$ be cofinal. We may assume that $X$ is a set of ordinals above $\Theta$ and let $\langle x_i\mid i<\omega\rangle$ be its enumeration of order type $\omega$.
        For each $i<\omega$, let $\mathcal{M}_i$ be the least initial segment of $\mathcal{M}_{\infty}^{\mathsf{b}}$ such that $x_i\in\ord\cap\mathcal{M}_i$ and $\rho(\mathcal{M}_i)=\Theta$.
        
        It is enough to show that for any $Y\in\power_{\omega_1}(\mathcal{M}^{\mathsf{b}}_{\infty})$, $Y\in L(\mathcal{M}_{\infty}^{\mathsf{b}}, X, \Gamma^*_g, \mathbb{R}^*_g)$. Fix such a $Y$ and its enumeration $\langle y_j\mid j<\omega\rangle$ of order type $\omega$. For each $j<\omega$, let $n_j<\omega$ be the least $n<\omega$ such that $y_j\in\mathcal{M}_n$. Since $\mathcal{M}_{n_j}$ is sound, there are a formula $\phi$ and $s\in [\Theta]^{<\omega}$ such that $y_{n_j}$ is the unique $y$ such that
        \[
        \mathcal{M}_{n_j}\models\phi[y, s, p(\mathcal{M}_{n_j})],
        \]
        where $p(\mathcal{M}_{n_j})$ is the standard parameter of $\mathcal{M}_{n_j}$. Then let $\phi_j$ and $s_j$ be the least $\phi$ and $s$ such that $\mathcal{M}_{n_j}\models\phi[y_{n_j}, s, p(\mathcal{M}_{n_j})]$. Now the set $\{\langle n_j, \phi_j, s_j\rangle\mid j<\omega\}$ can be coded into a countable subset $A$ of $\Theta$. Since $\Theta$ is regular in $W$, $A\subset\lambda$ for some $\lambda<\Theta$.
        Since $\Theta=\Theta^{L(\Gamma^*, \mathbb{R}^*_g)}$, there is a surjection $f\colon\mathbb{R}^*_g\to\lambda$ in $L(\Gamma^*_g, \mathbb{R}^*_g)$.
        Using such an $f$, $A$ can be coded into $\mathbb{R}^*_g$ and thus $A\in L(\Gamma^*_g, \mathbb{R}^*_g)$.
        Then $Y$ is definable over $\mathcal{M}_{\infty}^{\mathsf{b}}$ from the parameter $A$ in $L(\Gamma^*_g, \mathbb{R}^*_g)$, so $Y\in L(\mathcal{M}_{\infty}^{\mathsf{b}}, X, \Gamma^*_g, \mathbb{R}^*_g)$.
    \end{proof}
 
    We get the following lemma as in \cite{AD_ForcingAxiom_NSIdeal} and \cite{larson2021failures}.
	
	\begin{lem}\label{AC_in_Pmax_extension}
		$W[G\ast H]\models \mathsf{ZFC}+\mathsf{MM}^{++}(\mathfrak{c})$.\footnote{Here, $\mathsf{MM}^{++}(\mathfrak{c})$ denotes $\text{Martin's Maximum}^{++}$ for posets of size at most continuum.}
	\end{lem}
	\begin{proof}
        Using \cref{another_representation_of_W}, this lemma follows from the proof of \cite[Theorem 9.39]{AD_ForcingAxiom_NSIdeal}. We will give a detailed proof of how to get the $\mathsf{AC}$ in $W[G*H]$ here to make clear why \cref{another_representation_of_W} is helpful.
        
        Note that if $\mathsf{AD}^+$ holds and $\Theta$ is regular, then $\mathbb{P}_{\mathrm{max}}$ forces $\lvert\mathbb{R}\rvert=\omega_2$ and $\Theta=\omega_3$.
		Because $W=L(\mathcal{M}_{\infty}^{\mathsf{b}}, X, \Gamma^*_g, \mathbb{R}^*_g)$ where $\mathcal{M}_{\infty}^{\mathsf{b}}$ and $X$ are well-ordered by \cref{another_representation_of_W}, $\mathbb{R}^W=\mathbb{R}^*_g$, and $\power(\mathbb{R})^{W}=\Gamma^*_g$, we only need to show that
        \[
        W[G*H]\models\power(\omega_2)\text{ is well-ordered.}
        \]
        In the rest of the proof, we write $\omega_2=\omega_2^{W[G]}$ and $\omega_3=\omega_3^{W[G]}$.
  
        It is easy to see that $\mathrm{Add}(\omega_3,1)^{W[G]}$ adds a well-order of $\power(\omega_2)^{W[G]}$ of length $\omega_3$, because any binary sequence of length $\omega_2$ will eventually appear in the added generic function $\bigcup H\colon\omega_3\to 2$ by density argument: For any $f\colon\omega_2\to 2$ in $W[G]$, the set 
        \[
        \{p\in\mathrm{Add}(\omega_3,1)^{W[G]}\mid\exists\alpha<\omega_3\forall\xi<\omega_2(\alpha+\xi\in\mathrm{dom}(p) \land p(\alpha+\xi)=f(\xi))\}
        \]
        is dense, so in $W[G\ast H]$, we can order ${}^{\omega_2}2$ by sending each $f\in {}^{\omega_2}2$ to the least $\alpha<\omega_3$ such that $\forall\xi<\omega_2 ((\bigcup H)(\alpha+\xi)=f(\xi))$.
		
        We want to show that $\power(\omega_2)^{W[G]}=\power(\omega_2)^{W[G*H]}$, which completes the proof. While it follows from $\mathsf{ZF}$ that a ${<}\omega_3$-distributive poset does not add any subsets of $\omega_2$ and $\mathrm{Add}(\omega_3,1)$ is ${<}\omega_3$-closed, some choice principle is necessary to prove that ${<}\omega_3$-closure implies ${<}\omega_3$-distributivity for $\mathrm{Add}(\omega_3,1)$. However, by \cref{another_representation_of_W}, \cite[Theorem 9.36]{AD_ForcingAxiom_NSIdeal} implies that $\omega_2\mathchar`-\mathsf{DC}$ holds in $W[G]$ and this is enough for us as shown below:
		Let $f\colon\omega_2\to W[G]$ in $W[G\ast H]$ with a name $\dot{f}$. We may assume that $\emptyset\Vdash\dot{f}\colon\omega_2\to W[G]$. For each $p\in\mathrm{Add}(\omega_3,1)^{W[G]}$, let $f_p\in W[G]$ be the largest initial segment of $f$ decided by $p$. Namely, $f_p$ is a function such that $p\Vdash \check{f_p}\subset\dot{f}$ and for any function $g\in W[G]$ with $\mathrm{dom}(f_p)\subsetneq\mathrm{dom}(g)$, $p\not\Vdash\check{g}\subset\dot{f}$. Define a relation $\prec$ on $\mathrm{Add}(\omega_3,1)^{W[G]}$ by $p\prec q$ if $p\leq q$ and $f_p\supsetneq f_q$. If there is $p\in\mathrm{Add}(\omega_3, 1)^{W[G]}$ that decides $f$, then $f\in W[G]$ and we are done. Now suppose otherwise. Then for any $\succ$-sequence of conditions $\vec{p}$ of length $<\omega_2$, we can get a condition $q$ stronger than any conditions in the sequence by ${<}\omega_3$-closedness of $\mathrm{Add}(\omega_3,1)^{W[G]}$, and then take a condition $r\leq q$ that decides larger initial segment of $\dot{f}$ than $\bigcup_{\alpha<\mathrm{lh}(\vec{p})}f_{\vec{p}(\alpha)}$. By $\omega_2\mathchar`-\mathsf{DC}$ in $W[G]$, there is a $\succ$-sequence of length $\omega_2$. By ${<}\omega_3$-closedness again, we can take a condition stronger than any conditions in the sequence, which would decide $f$. Contradiction!
	\end{proof}

    We will use the following consequence of $\mathsf{MM}(\mathfrak{c})$ proved by Woodin in \cite{Woodin_equiv}.

    \begin{thm}[Woodin, \cite{Woodin_equiv}]\label{AD+_conjecture}
        Asssume that $\mathsf{ZFC}+\mathsf{MM}(\mathfrak{c})$ holds. Then the $\mathsf{AD}^+$ conjecture holds: Let $A_0, A_1\subset\mathbb{R}$ be such that $L(A_i, \mathbb{R})$ and let $\Delta_i$ be the Suslin-co-Suslin sets of $L(A_i, \mathbb{R})$.
        Suppose that any $B\in\Delta_0\cup\Delta_1$ is ${<}\omega_2$-universally Baire. Then
        \[
        L(\Delta_0\cup\Delta_1, \mathbb{R})\models\mathsf{AD}^+.
        \]
    \end{thm}
 
    \begin{lem}\label{HOD_does_not_change}
		$\HOD^W=\HOD^{W[G\ast H]}$.
	\end{lem}
	\begin{proof}
        We have $\HOD^{W[G\ast H]}\subset\HOD^{W}$ because of the weak homogeneity of $\mathbb{P}_{\mathrm{max}}$ and $\mathrm{Add}(\omega_3,1)$. (It is a general fact that if a poset $\mathbb{P}$ is weakly homogeneous in $\HOD$, then $\HOD$ in a generic extension via $\mathbb{P}$ is contained in $\HOD$ of the ground model.) To show the other direction, it is enough to see that $W$ is ordinal definable in $W[G\ast H]$.

        \begin{claim}
			$\mathbb{R}^W = \mathbb{R}^{W[G*H]}$ and $\power_{\omega_1}(\mathcal{M}_{\infty}^{\mathsf{b}})^W =\power_{\omega_1}(\mathcal{M}_{\infty}^{\mathsf{b}})^{W[G\ast H]}$.
		\end{claim}
		\begin{proof}
            Both equations immediately follows from the fact that $G\ast H$ is generic for a countably closed poset.
        \end{proof}
        
		\begin{claim}\label{define_P(R)}
		  $\power(\mathbb{R})^W$ is ordinal definable in \(W[G*H]\).
		\end{claim}
		\begin{proof}
            For each $A$, let $\Delta_A$ be the set of Suslin-co-Suslin sets of $L(A, \mathbb{R})$.
            Let
            \begin{multline*}
            \Gamma=\bigcup\{\Delta_A\mid\exists A\subset\mathbb{R}^{W[G*H]} (\text{any set in $\Delta_A$ is $<\omega_2$-universally Baire and}\\
            L(A, \mathbb{R})^{W[G\ast H]}\models\mathsf{AD}^+)\}.
            \end{multline*}
            Then we can show that $\power(\mathbb{R})^W\subset\Gamma$ as follows.
            Let $B\in\power(\mathbb{R})^W$.
            As $W\models\mathsf{AD}_{\mathbb{R}}$ holds, there is an $A\in\power(\mathbb{R})^W$ such that $B\in\Delta_A$.
            Also, the fourth author showed in \cite{ThetaUB} that $\mathsf{AD}_{\mathbb{R}}$ implies that all sets of reals are $X$-universally Baire if there is a surjection from $\mathbb{R}$ onto $X$.
            Since $\mathrm{Add}(\omega_3, 1)^{W[G]}$ does not add any subset of $\omega_2$, $B$ is still $<\omega_3$-universally Baire in $W[G*H]$.
            Therefore $B\in\Gamma$.
            
            Since $\Theta^W=\Theta^{W[G\ast H]}$, any new set of reals cannot be Wadge compatible with sets of reals in $W$.
            By \cref{AD+_conjecture}, $W[G*H]$ satisfies $\mathsf{AD}^+$ conjecture, so $\Gamma=\power(\mathbb{R})^W$.
            Obviously $\Gamma$ is ordinal definable in $W[G*H]$.
		\end{proof}

        \begin{claim}\label{ordinal_definability_of_M_infty}
			$\mathcal{M}_{\infty}^{\mathsf{b}}$ is ordinal definable in $W[G\ast H]$.
		\end{claim}
		\begin{proof}
            First, $\mathcal{M}_{\infty}\vert\kappa_{\infty}=(\HOD\|\Theta)^W = (\HOD\|\Theta)^{W[G*H]}$ as in the proof of \cref{HOD_in_CDM}.
            Note that $\mathcal{M}_{\infty}^{\mathsf{b}}$ is a stack of all sound lbr hod premice $\mathcal{M}$ such that $\mathcal{M}_{\infty}\vert\kappa_{\infty}\triangleleft\mathcal{M}, \rho(\mathcal{M})=\kappa_{\infty}$, and whenever $\pi\colon\mathcal{N}\to\mathcal{M}$ is elementary and $\mathcal{N}$ is countable, there is an $\omega_1$-iteration strategy $\Lambda$ for $\mathcal{N}$ such that $\Sigma^{\pi}_{\mathcal{M}_{\infty}\vert\kappa_{\infty}}\subset\Lambda$ and $\mathrm{Code}(\Lambda)\in\power(\mathbb{R})^W$.
            By \cref{define_P(R)}, $\mathcal{M}_{\infty}^{\mathsf{b}}$ is ordinal definable in $W[G*H]$.
		\end{proof}

        These claims give a definition of $W$ inside of $W[G*H]$ using only ordinal parameters.
	\end{proof}

    \renewcommand{\proofname}{Proof of \cref{main_thm_on_Pmax}}
	\begin{proof}
        We already showed that $W[G\ast H]\models\mathsf{ZFC}$ in \cref{AC_in_Pmax_extension}.
        To show the desired property of $\omega_1, \omega_2$ and $\omega_3$ in $W[G\ast H]$, note that any club subset of $\kappa\cap\mathrm{Cof}(\omega)$ in $W$ is still a club subset of $\kappa\cap\mathrm{Cof}(\omega)$ in $W[G\ast H]$ by countable completeness: any counterexample to being an $\omega$-club would be a new $\omega$-sequence, but $G*H$ adds no such sequences. Note that this argument uses $\mathsf{DC}$ in $W$.

        Let $i\in\{1, 2\}$. In $W$, since $\mathsf{AD}$ holds, the club filter on $\omega_i\cap\mathrm{Cof}(\omega)$ is an ultrafilter (cf.\ \cite[Theorem 33.12 (i)]{SetTheory_Jech} for $i=1$ and \cite[Corollary 5.20]{Intro_Combi_of_Determinacy} for $i=2$). Then because $\omega_i^W=\omega_i^{W[G\ast H]}$ and $\HOD^W=\HOD^{W[G*H]}$ holds by \cref{HOD_does_not_change}, in $W[G\ast H]$, the restriction of the club filter on $\omega_i\cap\mathrm{Cof}(\omega)$ to $\HOD$ is an ultrafilter in $\HOD$. Because $\Theta^W=\omega_3^{W[G\ast H]}$ and $\HOD^W=\HOD^{W[G*H]}$ holds, \cref{main_thm_on_CDM}(4) implies that in  $W[G\ast H]$, the restriction of the club filter on $\omega_3\cap\mathrm{Cof}(\omega)$ to $\HOD$ is an ultrafilter in $\HOD$.
	\end{proof}
    \renewcommand{\proofname}{Proof}
 
	This completes the proof of \cref{MainTheorem}.

    \section{Final Remark}

    Using our proof, one can obtain a model satisfying the conclusion of \cref{MainTheorem} directly from some determinacy theory.

    \begin{thm}\label{variant_of_main_theorem}
    Suppose that
    \begin{itemize}
        \item $\mathsf{AD}^+ +\mathsf{AD}_{\mathbb{R}}$,
        \item There is an $\mathbb{R}$-complete normal measure on $\Theta$,
        \item There is a surjection from $\Theta$ onto $\power(\Theta)\cap\HOD$, and
        \item $\mathsf{HPC}$ (Hod Pair Capturing) holds.
    \end{itemize}
    Then there is a transitive model $W$ of $\mathsf{AD}^+$ containing $\ord\cup\mathbb{R}$ such that if $G*H\subset(\mathbb{P}_{\mathrm{max}}*\mathrm{Add}(\Theta, 1))^W$ is $W$-generic, then in $W[G*H]$, $\mathsf{ZFC}$ holds and for any $\kappa\in\{\omega_1, \omega_2, \omega_3\}$, the restriction of the club filter on $\kappa\cap\mathrm{Cof}(\omega)$ to $\HOD$ is an ultrafilter in $\HOD$.
    \end{thm}

    Note that the assumption of \cref{variant_of_main_theorem} is consistent relative to a Woodin limit of Woodin cardinals:
    Let $\mathcal{V}$ and $g$ be as in \cref{main_thm_on_CDM}. In $\mathcal{V}[g]$, we define
    \[
    \mathsf{CDM}[\mu]=L(\mathcal{M}_{\infty}, \cup_{\xi<\delta_{\infty}}{}^{\omega}\xi, \Gamma^*_g, \mathbb{R}^*_g)[\mu],
    \]
    where $\mu$ is the club filter on $\Theta^{\mathsf{CDM}}\cap\mathrm{Cof}(\omega)$. Then one can show that $\mathsf{CDM}[\mu]$ satisfies the assumption of \cref{variant_of_main_theorem}.

    We only give a proof outline of \cref{variant_of_main_theorem}. By the $\HOD$ computation up to $\Theta$, $\HOD\|\Theta$ can be represented as a direct limit of lbr hod mice, so let $\mathcal{H}$ be such representation. Also, the direct limit system gives the canonical iteration strategy $\Sigma$ for $\mathcal{H}$. Let $\mathcal{H}^+$ be the stack of all sound $\Sigma$-premice $\mathcal{M}$ over $\mathcal{H}$ such that
	\begin{itemize}
		\item $\rho(\mathcal{M})=\Theta$, and
		\item $\mathcal{M}$ is countably iterable in the following sense: every countable transitive $\mathcal{N}$ embeddable into $\mathcal{M}$ via $\pi$ has an $\omega_1$-iterable as a $\Sigma^\pi$-premouse over $\pi^{-1}(\mathcal{H})$.
	\end{itemize}
	We let
    \[
    W=L(\mathcal{H}^+, \power_{\omega_1}(\mathcal{H}^+), \power(\mathbb{R})).
    \]
    Then one can show that $\mathcal{H}^+ = (\HOD\|\Theta^+)^W$.
    The existence of strongly condensing sets, which is a countable subset of $\mathcal{H}^+$, is shown by the argument in \cite[Chapter 9]{LSAbook}. It is similar to the proof of \cref{strongly_condensing} but somewhat more involved.
    To do this, we need the following kind of failure of covering.

    \begin{lem}\label{failure_of_covering}
	   $\cf(\ord\cap\mathcal{H}^+)=\omega$.
	\end{lem}
	\begin{proof}
		Let $\lambda=\ord\cap\mathcal{H}^+$. We make use of the square principle in $\mathcal{H}^+$. We say that a \emph{$\square_{\Theta}$-sequence of $\mathcal{H}^+$} is a sequence $\langle C_{\alpha}\mid\alpha<\lambda\rangle$ such that for each $\alpha<\lambda$,
        \begin{itemize}
            \item $C_{\alpha}\subset\alpha$ is a club subset of $\alpha$,
            \item for each limit point $\beta$ of $C_{\alpha}$, $C_{\beta}=C_{\alpha}\cap\beta$, and
            \item the order type of $C_{\alpha}$ is at most $\kappa$.
        \end{itemize}
        We also say that the sequence $\langle C_{\alpha}\mid\alpha<\lambda\rangle$ is \emph{threadable} if there is a club $E\subset\lambda$ such that $C_{\alpha} = E\cap\alpha$ for each limit point $\alpha$ of $E$. The construction of a square sequence in \cite{SZsquare} shows that
        \begin{itemize}
            \item there is a $\square_{\Theta}$-sequence $\vec{C}=\langle C_{\alpha}\mid\alpha<\lambda\rangle$ of $\mathcal{H}^+$, and
            \item if $\cf(\lambda)>\omega$, then $\vec{C}$ is not threadable.
        \end{itemize}
        The second clause follows because a thread of a $\square_{\Theta}$-sequence of $\mathcal{H}^+$ is essentially a $\Sigma$-mouse $\mathcal{M}$ such that $\mathcal{H}\triangleleft\mathcal{M}$ and $\rho(\mathcal{M})=\Theta$ and $\mathcal{M}$ is countably iterable.
		
		\begin{claim}
			$\cf(\lambda)<\Theta$.
		\end{claim}
        \begin{proof}
        By the assumption of \cref{variant_of_main_theorem}, there are a bijection $f\colon\Theta\to\mathcal{H}^+$ and a normal $\mathbb{R}$-complete ultrafilter $\mu$ on $\Theta$. In \cite{HODmeas}, it is shown that $\mu$ is amenable to $\mathcal{H}^+$. Let
        \[
        W_0=L(\mathcal{H}^+, f)[\mu].
        \]
        Let $\mu_0=\mu\cap W_0\in W_0$ and let $\pi_{\mu_0}\colon W_0\to\ult(W_0, \mu_0)$ be the ultrapower map.
        Suppose that $\cf(\lambda)=\Theta$ and we shall reach a contradiction by defining a thread through $\vec{C}$ using $\pi_{\mu_0}$. Let $\pi_{\mu_0}(\vec{C})=\langle D_\alpha\mid\alpha<\pi_{\mu_0}(\lambda)\rangle$. Also, let $\eta=\sup\pi_{\mu_0}[\lambda]$ and let $E=\pi_{\mu_0}^{-1}[D_\eta]$. Since $\pi_{\mu_0}[\lambda]$ is an $\omega$-club and $\pi_{\mu_0}[\lambda]\in \ult(W_0, \mu_0)$, so $D_\eta\cap\pi_{\mu_0}[\lambda]\neq\emptyset$ and thus $E\neq\emptyset$. Let $\alpha$ be a limit point of $E$. Then $\pi_{\mu_0}(\alpha)\in D_\eta$ and thus $D_{\pi_{\mu_0}(\alpha)}=D_\eta\cap\pi_{\mu_0}(\alpha)$. It follows that $C_\alpha=E\cap\alpha$, which means that $E$ is a thread through $\vec{C}$.
        \end{proof}
  
		Now let $\tau=\cf(\lambda)<\Theta$ and let $g\colon\tau\to\lambda$ be cofinal. Let
        \[
        W_1=L(\mathcal{H}^+, g)[\mu].
        \]
        Let $\mu_1=\mu\cap {W_1}\in W_1$ and let $\pi_{\mu_1}\colon W_1\to\ult(W_1, \mu_1)$ be the ultrapower map.

        Let $h\colon\Theta\to\mathcal{H}$ be defined by $h(\kappa)=\mathcal{H}\vert(\kappa^+)^{\mathcal{H}}$. Note that $h\in W_1$. Then $\pi_{\mu_1}(h)(\Theta)=\pi_{\mu_1}(\mathcal{H})\vert(\Theta^+)^{\pi_{\mu_1}(\mathcal{H})}$ is a $\Sigma$-premouse over $\mathcal{H}$. Since $\mu$ is countably complete, $\pi_{\mu_1}(h)(\Theta)$ is countably iterable. Hence $\pi_{\mu_1}(h)(\Theta)\trianglelefteq\mathcal{H}^+$. 

        We in fact claim that $\pi_{\mu_1}(h)(\Theta)=\mathcal{H}^+$. Suppose not. We then have some $\mathcal{M}\trianglelefteq\mathcal{H}^+$ such that $\rho(\mathcal{M})=\Theta$ and $\pi_{\mu_1}(h)(\Theta)\triangleleft\mathcal{M}$. Notice now that $\mathcal{M}\in \pi_{\mu_1}(\mathcal{H}^+)$ as $\mathcal{M}$ is the transitive collapse of an appropriate fine structural hull of $\pi_{\mu_1}(\mathcal{M})$\footnote{E.g.\ if $\rho_1(\mathcal{M})=\Theta$ then $\mathcal{M}$ is the transitive collapse of the $\Sigma_1$-hull of $\pi_{\mu_1}(\mathcal{M})$ with parameters from $\Theta\cup\{p_1(\pi_{\mu_1}(\mathcal{M})\}$.}. It follows that $\pi_{\mu_1}(h)(\Theta)\triangleleft \pi_{\mu_1}(\mathcal{H})\vert(\Theta^+)^{\pi_{\mu_1}(\mathcal{H})}$, which is a contradiction. We thus have that $\lambda=(\Theta^+)^{\pi_{\mu_1}(\mathcal{H}^+)}$.
        
        So $\ult(W_1, \mu_1)\models\cf((\Theta^+)^{\pi_{\mu_1}(\mathcal{H})})=\tau$ as witnessed by $g\in \ult(W_1, \mu_1)$, and thus
        \[
        \{\kappa<\Theta\mid W_1\models\cf((\kappa^+)^{\mathcal{H}})=\tau\}\in\mu_1.
        \]
        By $\mathsf{HPC}$, $\cf((\kappa^+)^{\mathcal{H}})=\omega$ for $\mu$-almost all $\kappa$. Therefore, $\tau=\omega$.
	\end{proof}
 
    The rest of the argument is more or less the same as what we did in this paper, so we leave it to the readers.

\newcommand{\etalchar}[1]{$^{#1}$}

\end{document}